\documentclass[final]{siamltex}
\usepackage{epsfig}
\usepackage{amsmath}
\usepackage{amsfonts}
\usepackage{algorithm}
\usepackage{algorithmic}
\usepackage{graphicx}
\usepackage{subfigure}

\title{ A Subspace Method for Large-Scale Eigenvalue Optimization }
\author{
	Fatih Kangal\footnotemark[1]	
	\and Karl Meerbergen\footnotemark[2]
   	\and Emre Mengi\footnotemark[3]
   	\and Wim Michiels\footnotemark[2]
}

\begin{document}
\maketitle

\renewcommand{\thefootnote}{\fnsymbol{footnote}}
\footnotetext[1]{Department of Mathematics, Ko\c{c} University, Rumelifeneri Yolu, 34450 Sar\i yer-\.{I}stanbul, Turkey (fkangal@ku.edu.tr).}
\footnotetext[2]{Department of Computer Science, KU~Leuven, University of Leuven, 3001 Heverlee, Belgium 
(Karl.Meerbergen@cs.kuleuven.be, Wim.Michiels@cs.kuleuven.be).
The work of the authors were supported by OPTEC, the Optimization in Engineering Center of KU~Leuven, the Research Foundation Flanders (project G.0712.11N) and by the project UCoCoS, funded  by the 
European Union's Horizon 2020 research and innovation programme 
under the Marie Sklodowska-Curie Grant Agreement
No 675080.}
\footnotetext[3]{Department of Mathematics, Ko\c{c} University, Rumelifeneri Yolu, 34450 Sar\i yer-\.{I}stanbul, Turkey (emengi@ku.edu.tr).
The work of the author was supported in part by the European Commision grant
PIRG-GA-268355, TUBITAK - FWO (Scientific and Technological Research Council of Turkey -
Belgian Research Foundation, Flanders) joint grant 113T053, and by the BAGEP program of Turkish 
Academy of Science.}

\begin{abstract}
\noindent
We consider the minimization or maximization of the $J$th largest eigenvalue of 
an analytic and Hermitian matrix-valued function, and build on  
Mengi \emph{et al.} (2014, \emph{SIAM J. Matrix Anal. Appl.}, 35, 699-724).  
This work addresses the setting when the matrix-valued function involved is 
very large. We describe subspace procedures that convert the original problem into a small-scale 
one by means of orthogonal projections and restrictions to certain subspaces,
and that gradually expand these subspaces based on the optimal solutions of small-scale 
problems. Global convergence and superlinear rate-of-convergence results with respect to the dimensions of 
the subspaces are presented in the infinite dimensional setting, where  the matrix-valued function is 
replaced by a compact operator depending on parameters. In practice, it suffices to solve eigenvalue 
optimization problems involving matrices with sizes on the scale of tens, instead of the original 
problem involving matrices with sizes on the scale of thousands. 
\\

\\


\end{abstract}

\pagestyle{myheadings}
\thispagestyle{plain}
\markboth{ F. KANGAL, K. MEERBERGEN, E. MENGI, AND W. MICHIELS }{ LARGE-SCALE EIGENVALUE OPTIMIZATION }

\section{Introduction}
We are concerned with the global optimization problems 
\[
	\textbf{(MN)} \;\;
	{\rm minimize} \:
	\{ \lambda_J(\omega)	\; | \;	{\omega \in \Omega}	\}
			\;\;\;\;\;\;\;	{\rm and}		\;\;\;\;\;\;\;
	\textbf{(MX)} \;\;
	{\rm maximize} \:
	\{ \lambda_J(\omega)	\; | \;	{\omega \in \Omega}	\}.
\]
The feasible region $\Omega$ of these optimization problems is a compact subset of ${\mathbb R}^d$.
Furthermore, letting  $\ell^2({\mathbb N})$ denote the sequence space consisting of square summable infinite sequences
of complex numbers equipped with the inner product $\langle w, v \rangle = \sum_{k = 0}^\infty \overline{w_k} \: v_k$ as well as 
the norm $\| v \|_2 = \sqrt{  \sum_{k = 0}^\infty | v_k |^2  }$,
the objective function $\lambda_J(\omega)$ is the $J$th largest eigenvalue of a compact self-adjoint operator 
\begin{equation}\label{eq:mat_func}
	{\mathbf A}(\omega) : \ell^2({\mathbb N})	\rightarrow  \ell^2({\mathbb N}),	\quad
	{\mathbf A}(\omega)	:=	\sum_{\ell = 1}^\kappa		f_\ell(\omega)	{\mathbf A}_\ell		
\end{equation}
for every $\omega \in \overline{\Omega}$, an open subset of ${\mathbb R}^d$ containing the feasible region $\Omega$.
Above ${\mathbf A}_\ell : \ell^2({\mathbb N}) \rightarrow \ell^2({\mathbb N})$ and $f_\ell : \overline{\Omega} \rightarrow {\mathbb R}$ 
for $\ell = 1,\dots,\kappa$ represent given compact self-adjoint operators and real-analytic functions, respectively.
Throughout the text, ${\mathbf A}(\omega)$ for each $\omega$ and ${\mathbf A}_1, \dots, {\mathbf A}_\kappa$ as in (\ref{eq:mat_func})
could intuitively be considered as infinite dimensional Hermitian matrices.

Our interest in the infinite dimensional eigenvalue optimization problems \textbf{(MN)} and \textbf{(MX)}
rise from their finite dimensional counterparts, which for given Hermitian matrices $A_\ell \in {\mathbb C}^{n\times n}$ 
for $\ell = 1, \dots, \kappa$ involve the matrix-valued function
\begin{equation}\label{eq:mat_func2}
	A(\omega)	:= 	\sum_{\ell = 1}^\kappa		f_\ell(\omega)	A_\ell
\end{equation}
instead of the parameter dependent operator ${\mathbf A}(\omega)$. These problems come with
standard challenges due to their nonsmoothness and nonconvexity.
But we would like to tackle a different challenge, when the matrices in them are very large, that is $n$ is very large. 
Thus, the primary purpose of this paper is to deal with large-dimensionality, it does not address the inherent difficulties due to 
nonconvexity and nonsmoothness. We introduce the ideas in the idealized infinite dimensional setting, only because 
this makes a rigorous convergence analysis possible.

To deal with large dimensionality, we propose restricting the domain and projecting the range of the map 
$v \mapsto {\mathbf A}(\omega) v$ to small subspaces. This gives rise to eigenvalue optimization problems involving 
small matrices, which we call reduced problems. Two greedy procedures are presented here to construct small 
subspaces so that the optimal solution of the reduced problem is close to the optimal solution of the original problem. 
For both procedures, we observe a superlinear rate of decay in the error with respect to the subspace dimension. 
The first procedure is more straightforward and constructs smaller subspaces, shown to converge
at a superlinear rate when $d=1$, but lacks a complete formal argument 
justifying its quick convergence when $d \geq 2$. The second constructs larger subspaces, but comes with a formal 
proof of superlinear convergence for all $d$.

While the proposed procedures operate on \textbf{(MN)} and \textbf{(MX)} similarly, 
there are remarkable differences in their convergence behaviors in these two contexts. 
The proposed subspace restrictions and projections on the map $v \mapsto {\mathbf A}(\omega) v$ 
lead to global lower envelopes for $\lambda_J(\omega)$. These lower envelopes in turn make the
convergence to the globally smallest value of $\lambda_J(\omega)$ possible as the dimensions
of the subspaces by the procedures grow to infinity, which we prove formally. Such a global convergence 
behavior does not hold for the maximization problem: if the subspace dimension is let grow
to infinity, the globally maximal values of the reduced problems converge to a locally maximal
value of $\lambda_J(\omega)$, that is not necessarily globally maximal. 
But the maximization problem possesses a remarkable
low-rank property: there exists a $J$ dimensional subspace such that when the map
 $v \mapsto {\mathbf A}(\omega) v$ is restricted and projected to this subspace, the
 resulting reduced problem has the same globally largest value as $\lambda_J(\omega)$.
 The minimization problem does not enjoy an analogous low-rank property.

\subsection{Motivation}
Large eigenvalue optimization problems arise from various applications.
For instance, the distance to instability from a large stable matrix with respect to the matrix 2-norm
yields large singular value optimization problems \cite{VanLoan1985}, that can be converted
into large eigenvalue optimization problems. The computation of the H-infinity norm of the transfer function of a linear
time-invariant (LTI) control system can be considered as a generalization of the computation of the distance to instability. 
This is a norm for the operator that 
maps inputs of the LTI system into outputs, and plays a major role in robust control. The singular value optimization characterization for the H-infinity
norm involves large matrices if the input, output or, in particular, the intermediate state space have large dimension.
The state-of-the-art algorithms for H-infinity norm computation \cite{Boyd1990, Bruinsma1990} cannot 
cope with such large-scale control systems.
In engineering applications,  the largest eigenvalue of a matrix-valued function is often sought 
to be minimized. A particular application is the numerical scheme for the design of the strongest 
column subject to volume constraints \cite{Cox1992}, where the sizes of the matrices depend on the 
fineness of a discretization imposed on a differential operator. These matrices can be very 
large if a fine grid is employed. The standard semidefinite program (SDP) formulations received a lot of attention 
by the convex optimization community since the 1990s \cite{Vandenberghe1996}. They concern
the optimization over the cone of symmetric positive semidefinite matrices of a linear objective function 
subject to linear constraints.  The dual problem of an SDP under mild assumptions can be cast as an 
eigenvalue optimization problem \cite{Helmberg2000}. If the size of the matrix variable of an SDP is large, 
the associated eigenvalue optimization problem involves large matrices. The current SDP solvers are 
usually not suitable to deal with such large-scale problems.

\subsection{Literature}
Subspace projections and restrictions have been applied to particular eigenvalue optimization problems in the past.
But general procedures such as the ones in this paper have not been proposed and studied thoroughly. 
A subspace restriction idea has been employed specifically for the computation of the pseudospectral abscissa 
of a matrix in \cite{Kressner2014}. This computational problem involves the optimization of a linear objective function
subject to a constraint on the smallest singular value of an affine matrix-valued function. Fast convergence is observed 
in that work, and confirmed with a superlinear rate-of-convergence result.  
In the context of standard semidefinite programs, subspace methods have been used for quite a while, in particular
the spectral bundle method \cite{Helmberg2000} is based on subspace ideas. The small-scale optimization problems 
resulting from subspace restrictions and projections are solved by standard SDP solvers \cite{Vandenberghe1996}.
A thorough convergence analysis for them has not been performed, also their efficiency is not fully realized
in practice. Extensions for convex quadratic SDPs \cite{Lin2012} and linear matrix inequalities \cite{Miller2000}
have been considered. All these large-scale problems connected to SDPs are convex. 
We present unified procedures and their convergence analyses that are applicable regardless of whether 
the problem is convex or nonconvex.


\subsection{Spectral Properties and Operator Norm of ${\mathbf A}(\omega)$}
The spectrum of ${\mathbf A}(\omega)$ is defined by
\[
	\sigma	\left( {\mathbf A}(\omega) \right)	\;	:=	\; 	\{ z \in {\mathbb R}	\; |	\; zI - {\mathbf A}(\omega) \; \text{ is not invertible}	\}.
\]
This set contains countably many real eigenvalues each with a finite multiplicity, also the only accumulation 
point of these eigenvalues is 0 (see \cite[page 185, Theorem 6.26]{Kato1995}). We assume, throughout the text, that
${\mathbf A}(\omega)$ has at least $J$ positive eigenvalues for all $\omega \in \Omega$\footnote[3]{Recall that our set-up is motivated by the finite-dimensional 
case (large matrices).  There the corresponding assumption can be made without loosing generality: instead of the optimization of the $J$th largest eigenvalue 
of $A(\omega)$, one may equivalently consider the optimization of the $J$th largest eigenvalue of $A(\omega) + \tau I$ for $\tau\geq 0$ sufficiently large,  
since the added term only introduces a shift of $\tau$ of the eigenvalues, of both the original problems and the problems obtained by orthogonal projection.}.
This ensures that $\lambda_J(\omega)$ is well-defined over $\Omega$.


The eigenspace associated with each eigenvalue of ${\mathbf A}(\omega)$ is also finite dimensional \cite[Corollary 6.44]{Kubrusly2001}. 
Furthermore, the set of eigenvectors of ${\mathbf A}(\omega)$ can be chosen in a way so that they are orthonormal
and complete in $\ell^2({\mathbb N})$ \cite[Theorem 4.2.23]{Davies2007}, i.e., there exist an orthonormal sequence $\{ v^{(j)} \}$ in $\ell^2({\mathbb N})$
and a sequence $\{ \mu^{(j)} \}$ of real numbers such that ${\mathbf A}(\omega) v^{(j)} = \mu^{(j)} v^{(j)}$ for all $j = 1,2,\dots$, 
$\mu^{(j)} \rightarrow 0$ as $j \rightarrow \infty$ and any vector $v \in \ell^2({\mathbb N})$ can be expressed as 
$v = \sum_{j = 1}^\infty \alpha_j v^{(j)}$ for some scalars $\alpha_j$.

Since ${\mathbf A}_1, \dots, {\mathbf A}_\kappa$ and ${\mathbf A}(\omega)$ are compact, they are bounded.
Hence, their operator norms 
\begin{align*}
	\| {\mathbf A}_\ell \|_2 \; := \; {\rm sup} \left\{ \| {\mathbf A}_\ell v \|_2	\; | \; v \in \ell^2({\mathbb N}) \text{ such that } \| v \|_2 = 1	\right\}
	\text{  for } \ell = 1, \dots, \kappa		&	\;\;\; \text{  and  }	\\
	 \| {\mathbf A}(\omega) \|_2 \; := \; {\rm sup} \left\{ \| {\mathbf A}(\omega) v \|_2	\; | \; v \in \ell^2({\mathbb N}) \text{ such that } \| v \|_2 = 1	\right\}	
	 \;\;\;\;\;\;\;\;\;\;\;\;\;\;\;\;\;\;   &
\end{align*}
are also well-defined.

\subsection{Outline}
We formally define the reduced problems, then analyze the relation between the original 
and the reduced problems in the next section. Some of this analysis, in particular the low-rank
property discussed above, apply only to the maximization problem. Last two subsections
of the next section are devoted to the introduction of the two subspace procedures for
eigenvalue optimization. Interpolation properties between the original problem and the
reduced problems formed by these two subspace procedures are also investigated there.
Section \ref{sec:convergence} concerns the convergence analysis of the two subspace procedures.
In particular, Section \ref{sec:global_convergence} establishes the convergence of the subspace procedures
globally to the smallest value of $\lambda_J(\omega)$ for the minimization problem as the subspace dimensions
grow to infinity. In practice, we observe at least a superlinear rate-of-convergence with respect to the dimension
of the subspaces for both of the procedures. We prove this formally in Section \ref{sec:rate_of_convergence}
fully for one of the procedures and partly for the other. Section \ref{sec:var_ext} focuses on variations and extensions of the two
subspace procedures for eigenvalue optimization. Specifically, in Section \ref{sec:subspace_nopast} we argue 
that a variant that disregards the subspaces formed in the past iterations works effectively for the maximization 
problem but not for the minimization problem, in Section \ref{sec:singular_vals} and  \ref{sec:small_singular_vals} 
we extend the procedures to optimize a specified singular value of a compact operator
depending on several parameters analytically, and in Section \ref{sec:cutting_plane} we provide a comparison of
one of the subspace procedures for eigenvalue optimization with the cutting plane method \cite{Kelley1960}, 
which also constructs global lower envelopes repeatedly. Section \ref{sec:applications} 
describes the MATLAB software accompanying this work, and
the efficiency of the proposed subspace procedures on particular applications, namely
the numerical radius, the distance to instability from a matrix, and the minimization of the largest 
eigenvalue of an affine matrix-valued function. The text concludes with a summary
as well as research directions for future in Section \ref{sec:conclusion}.


\section{The Subspace Procedures} \label{sec:eig_subspace}
Let ${\mathcal V}$ be a finite dimensional subspace of $\ell^2({\mathbb N})$, and $V := \{v_1, \dots, v_m \}$
be an orthonormal basis for ${\mathcal V}$. The linear operator 
\begin{equation}\label{eq:change_coor}
	{\mathbf V} : {\mathbb C}^m \rightarrow \ell^2({\mathbb N}),	\quad\quad
	{\mathbf V} \alpha	:=	\sum_{j=1}^m \alpha_j v_j
\end{equation}
maps the coordinates of a vector $v \in {\mathcal V}$ relative to $V$
to itself. On the other hand, its adjoint ${\mathbf V}^\ast :  \ell^2({\mathbb N}) \rightarrow {\mathbb C}^m$ projects 
a vector in $\ell^2({\mathbb N})$ orthogonally onto ${\mathcal V}$ and represents the orthogonal projection 
in coordinates relative to $V$. The procedures that will be introduced in this section are based on
operators of the form
\begin{equation}\label{eq:reduced_mat_fun}
	{\mathbf A}^{{\mathcal V}}(\omega) := {\mathbf V}^\ast {\mathbf A}(\omega) {\mathbf V},
\end{equation}
which is the restriction of ${\mathbf A}(\omega)$ that acts only on ${\mathcal V}$, with its input and output
represented in coordinates relative to $V$. This operator can be expressed of the form
\[
	{\mathbf A}^{{\mathcal V}}(\omega)	=	\sum_{\ell = 1}^\kappa		f_\ell(\omega)	{\mathbf A}^{{\mathcal V}}_\ell,
	\;\;\;\;\;	 {\rm where} \;\;\;
	{\mathbf A}^{{\mathcal V}}_\ell := {\mathbf V}^\ast {\mathbf A}_\ell {\mathbf V},
\]
which is beneficial from computational point of view, because it is possible to form the $m\times m$
matrix representations of the operators ${\mathbf A}^{{\mathcal V}}_\ell$ in advance. 
The reduced eigenvalue optimization problems are defined in terms of the $J$th largest eigenvalue
$\lambda_J^{\mathcal V}(\omega)$ of ${\mathbf A}^{\mathcal V}(\omega)$ as
\begin{equation}\label{eq:reduced_problems}
	{\rm minimize} \:
	\{ \lambda_J^{\mathcal V}(\omega)	\; | \;	{\omega \in \Omega}	\}
			\;\;\;\;\;\;\;	{\rm and}		\;\;\;\;\;\;\;
	{\rm maximize} \:
	\{ \lambda_J^{\mathcal V}(\omega)	\; | \;	{\omega \in \Omega}	\}.
\end{equation}
The discussions above generalize when ${\mathcal V}$ is infinite dimensional.
In the infinite dimensional setting, the operators ${\mathbf A}^{\mathcal V}(\omega)$ and 
${\mathbf A}^{{\mathcal V}}_\ell$ are defined similarly in terms of an infinite countable
orthonormal basis $V = \{ v_j \; | \; j \in {\mathbb Z}^+ \} $ for ${\mathcal V}$ and the associated 
operator 
\begin{equation}\label{eq:change_coor2}
	{\mathbf V} \alpha	:=	\sum_{j=1}^\infty \alpha_j v_j.
\end{equation}

The eigenvalue function $\lambda^{\mathcal V}_J(\omega)$ is Lipschitz continuous, indeed there exists a
uniform Lipschitz constant $\gamma$ for all subspaces ${\mathcal V}$ of $\ell^2({\mathbb N})$ as formally stated 
and proven in the following lemma.
\begin{lemma}[Lipschitz Continuity]\label{lemma:Lipschitz_continuity}
	There exists a real scalar $\gamma > 0$ such that for all subspaces ${\mathcal V}$ of $\ell^2({\mathbb N})$,
	we have
	\[
		\left| \lambda^{\mathcal V}_J(\widetilde{\omega})		-	\lambda^{\mathcal V}_J(\omega) \right|
					\;	\leq	\;
		\gamma \cdot \| \widetilde{\omega} - \omega \|_2
		\quad
		\forall \widetilde{\omega}, \: \omega \in \overline{\Omega}.
	\]
\end{lemma}
\begin{proof}
By Weyl's theorem (see \cite[Theorem 4.3.1]{Horn1985} for the finite dimensional case;
its extension to the infinite dimensional setting is straightforward by exploiting the maximin
characterization (\ref{eq:maximin}) of $\lambda^{\mathcal V}_J(\omega)$ given below)
\begin{eqnarray*}
	\left| \lambda^{\mathcal V}_J(\widetilde{\omega})		-	\lambda^{\mathcal V}_J(\omega) \right|
				\; & \leq & \;
	\| {\mathbf A}^{\mathcal V}( \widetilde{\omega} ) - {\mathbf A}^{\mathcal V}( \omega )	\|_2	\\
				\; & \leq & \;
	\sum_{\ell=1}^\kappa \left| f_\ell( \widetilde{\omega})  -  f_\ell(\omega) \right|  \| {\mathbf A}^{\mathcal V}_\ell \|_2
	\quad
	\forall \widetilde{\omega}, \: \omega \in \overline{\Omega}.
\end{eqnarray*}
In the last summation $\| {\mathbf A}^{{\mathcal V}}_\ell \|_2 \leq \| {\mathbf A}_\ell \|_2$ and the real analyticity of
$f_\ell(\omega)$ implies its Lipschitz continuity, hence the existence of a constant $\gamma_\ell$ satisfying
\[
	\left| f_\ell(\widetilde{\omega}) - f_\ell(\omega) \right|
		\; \leq \;
	\gamma_\ell \|\widetilde{\omega} - \omega \|_2
	\quad
	\forall \widetilde{\omega}, \: \omega \in \overline{\Omega}.
\]
Combining these observations we obtain
\[
	\left| \lambda^{\mathcal V}_J(\widetilde{\omega})		-	\lambda^{\mathcal V}_J(\omega) \right|
				\; \leq \;
		\left( \sum_{\ell = 1}^\kappa  \gamma_\ell \| {\mathbf A}_\ell \|_2 \right) \cdot \| \widetilde{\omega} - \omega \|_2
		\quad
		\forall \widetilde{\omega}, \: \omega \in \overline{\Omega}
\]
as desired.
\end{proof}

Throughout the rest of this section, we first investigate the relation between $\lambda_J(\omega)$ 
and $\lambda_J^{\mathcal V}(\omega)$ as well as their globally maximal values. Some of these
theoretical results will be frequently used in the subsequent sections. The second and third parts
of this section introduce two subspace procedures for the generation of small dimensional
subspaces ${\mathcal V}$, leading to reduced eigenvalue optimization problems that approximate
the original eigenvalue optimization problems accurately.

\subsection{Relations between $\lambda_J^{\mathcal V}(\omega)$ and $\lambda_J(\omega)$} 
We start with a result about the monotonicity of $\lambda_J^{\mathcal V}(\omega)$ with respect to ${\mathcal V}$. 
This result is an immediate consequence of the following maximin characterization \cite[pages 1543-1544]{Dunford1998} 
(see also \cite[Theorem 4.2.11]{Horn1985} for the finite dimensional case) of $\lambda^{\mathcal V}_J (\omega)$
for all subspaces ${\mathcal V}$ of $\ell^2({\mathbb N})$:
\begin{equation}\label{eq:maximin}
	\lambda^{\mathcal V}_J (\omega)
				\;\; =  \;\;
	\sup_{{\mathcal S} \subseteq {\mathcal V}, \; \dim {\mathcal S} = J}	\;\;
	\min_{v \in {\mathcal S}, \; \| v \|_2 = 1}	\;	\langle {\mathbf A}(\omega) v, v \rangle,
\end{equation}
where $\langle \cdot, \cdot \rangle$ stands for the standard inner product 
$\langle w, v \rangle := \sum_{k=0}^\infty \overline{w_k} \: v_k$ and the outer supremum is 
over all $J$ dimensional subspaces of ${\mathcal V}$. Furthermore, if ${\mathbf A}^{\mathcal V}(\omega)$
has at least $J$ positive eigenvalues, that is if $\lambda_J^{\mathcal V}(\omega) > 0$, or if ${\mathcal V}$ is
finite dimensional, then the outer supremum in (\ref{eq:maximin}) is attained, hence can be replaced by maximum.
The monotonicity result presented next will play a central role later when we analyze the convergence of the subspace procedures.
\begin{lemma}[Monotonicity]\label{lemma:monotonicity}
Let ${\mathcal V}_1$, ${\mathcal V}_2$ be subspaces of $\ell^2 ({\mathbb N})$ of dimension larger 
than or equal to $J$ such that ${\mathcal V}_1  \;\; \subseteq \;\; {\mathcal V}_2$.  The following holds:
\begin{equation}\label{eq:monotone}
	\lambda^{{\mathcal V}_1}_J (\omega)	\;\;	\leq	\;\;		\lambda^{{\mathcal V}_2}_J (\omega)	\;\;	\leq	\;\;	\lambda_J(\omega).
\end{equation}
\end{lemma}
\vskip -2ex
\noindent
A consequence of this monotonicity result is the following interpolatory property.
\begin{lemma}[Interpolatory Property]\label{lemma:interpolatary}
Let ${\mathcal V}$ be a subspace of  $\ell^2 ({\mathbb N})$ of dimension larger than or equal to $J$, and ${\mathbf V}$
be the operator defined as in (\ref{eq:change_coor}) or (\ref{eq:change_coor2})
in terms of an orthonormal basis $V$ for ${\mathcal V}$. 
If $S := {\rm span}\{ s_1, \dots, s_J \} \; \subseteq \; {\mathcal V}$ where $s_1, \dots, s_J$ are eigenvectors corresponding to
the $J$ largest eigenvalues of ${\mathbf A}(\omega)$, then the following hold:
\begin{enumerate}
	\item[\bf (i)]  $\lambda_J(\omega) = \lambda^{\mathcal V}_J(\omega)$;
	\item[\bf (ii)] ${\mathbf V}^\ast s_J$ is an eigenvector of ${\mathbf V}^\ast {\mathbf A}(\omega) {\mathbf V}$ corresponding to 
	its eigenvalue $\lambda^{\mathcal V}_J(\omega)$.
\end{enumerate}
\end{lemma}
\begin{proof}
\textbf{(i)} 
We assume each $s_\ell \in {\mathcal V}$ for $\ell = 1, \dots, J$ is of unit length without loss of generality. 
It follows that there exist  $\alpha_1, \dots, \alpha_J$ of unit length such that $s_\ell = {\mathbf V} \alpha_\ell$ for $\ell = 1,\dots,J$. 
Now define ${\mathcal S}_\alpha := {\rm span} \{ \alpha_1, \dots, \alpha_J \} \subseteq {\mathbb C}^m$, and observe
\begin{eqnarray*}
	\lambda_J \left( 	\omega		\right)	\;  =  \; 
	\langle    {\mathbf A} \left(	\omega	\right) s_J,  s_J \rangle & = &
	\langle  {\mathbf A} \left( \omega \right)   {\mathbf V} \alpha_J , {\mathbf V}\alpha_J  \rangle  \\
	& = & \langle  {\mathbf V}^\ast {\mathbf A} \left( \omega \right)   {\mathbf V} \alpha_J , \alpha_J  \rangle \\	
	& = &	\min_{\alpha \in {\mathcal S}_\alpha, \| \alpha \|_2 = 1}	\; 	\langle {\mathbf V}^\ast    {\mathbf A} \left( \omega \right)  {\mathbf V} \alpha, \alpha \rangle
	\;\;     \leq 	 \;\;   \lambda^{\mathcal V}_J \left(  \omega \right).	
\end{eqnarray*}
The opposite inequality is immediate from Lemma \ref{lemma:monotonicity}, so
$
	\lambda_J \left(	\omega 	\right) 
			=
	\lambda^{\mathcal V}_J \left(  \omega \right)
$
as claimed. \\
\textbf{(ii)}
The equalities 
\[
	 \langle {\mathbf V}^\ast {\mathbf A} \left( \omega \right)  {\mathbf V}  \alpha_J, \alpha_J 	\rangle 	
	\;\; 	= 	\;\; 	\min_{\alpha \in \mathcal S_{\alpha}, \| \alpha \|_2 = 1}	\; 	\langle  {\mathbf V}^\ast {\mathbf A} \left( \omega \right)  {\mathbf V} \alpha, \alpha \rangle
	\;\;     = 	\;\;   \lambda^{\mathcal V}_J \left(  \omega \right).
\]
imply that $\alpha_J = {\mathbf V}^\ast s_J$ is an eigenvector of ${\mathbf V}^\ast {\mathbf A} \left(	  \omega 	\right) {\mathbf V}$ corresponding to the
eigenvalue $\lambda^{\mathcal V}_J \left(	  \omega 	\right)$.
\end{proof}

Maximization of the $J$th largest eigenvalue over a low-dimensional subspace is motivated by the next result.
According to the result, it suffices to perform the optimization on a proper $J$ dimensional subspace, which
is hard to determine in advance. Here and elsewhere, 
$\arg \max_{\omega \in \Omega}\; \lambda_J \left( \omega \right)$ and 
$\arg \min \in_{\omega \in \Omega}\; \lambda_J \left( \omega \right)$ denote
the set of global maximizers and global minimizers of $\lambda_J(\omega)$ over $\omega \in \Omega$, respectively.
\begin{lemma}[Low-Rank Property of Maximization Problems]\label{thm:low_rank}
For a given subspace ${\mathcal V} \subset \ell^2({\mathbb N})$ with dimension larger than or equal to $J$, 
consider the following assertions:
\begin{enumerate}
	\item[\bf (i)]
		$ \max_{\omega \in \Omega} \; \lambda^{\mathcal V}_J \left( \omega \right) \;\;		=		\;\; 
			\max_{\omega \in \Omega} \; \lambda_J \left( \omega \right) $;
	\item[\bf (ii)] ${\mathcal S}_\ast := {\rm span} \left\{ s_1, \dots, s_J \right\} \; \subseteq \; {\mathcal V}$, 
	where $s_1, \dots, s_J$ are eigenvectors corresponding to the $J$ largest eigenvalues of 
	${\mathbf A}(\omega_\ast)$ at some $\omega_\ast \in \arg \max_{\omega \in \Omega}\; \lambda_J \left( \omega \right) $.
\end{enumerate}
Assertion \textbf{(ii)} implies assertion \textbf{(i)}. Furthermore, when $J = 1$, assertions \textbf{(i)} and \textbf{(ii)} are equivalent.
\end{lemma}

\begin{proof}
Suppose $s_1, \dots, s_J \in {\mathcal V}$ satisfies assertion (ii). By Lemma \ref{lemma:interpolatary}, we have
\begin{eqnarray*}
\max_{\omega \in \Omega} \; \lambda_J \left( \omega \right) \;\; =	\;\;
\lambda_J \left( \omega_\ast \right)  \;\; = \;\; 	 \lambda^{\mathcal V}_J \left( \omega_\ast \right)
\;\;\;	\leq		\;\;\;	\max_{\omega \in \Omega} \;  \lambda^{\mathcal V}_J \left(  \omega \right).
\end{eqnarray*}
Lemma \ref{lemma:monotonicity} implies the opposite inequality,  proving assertion (i).

To prove that the assertions are equivalent when $J = 1$, assume that assertion (i) holds. 
Letting
$\omega_\ast$ 
be any point in 
$
	\arg \max_{\omega \in \Omega} \; \lambda^{\mathcal V}_1 \left( \omega \right),
$
denoting by ${\mathbf V}$ an operator defined as in (\ref{eq:change_coor}) or (\ref{eq:change_coor2})
in terms of an orthonormal basis $V$ for ${\mathcal V}$, and denoting by  $\alpha$ a unit eigenvector 
corresponding to the largest eigenvalue of ${\mathbf V}^\ast {\mathbf A}(\omega_\ast) {\mathbf V}$
(note that $\lambda_1^{\mathcal V}(\omega_\ast) = \lambda_1(\omega_\ast) > 0$, so the
unit eigenvector $\alpha$ is well-defined), we have
\begin{equation}\label{eq:eig_inequality}
	\max_{\omega \in {\mathbb R}^d} \; \lambda_1 \left( \omega \right)
					\;  =	\;
	\lambda^{\mathcal V}_1 \left( \omega_\ast  \right)
					\;  =  \;
	\langle  {\mathbf V}^\ast {\mathbf A}(\omega_\ast) {\mathbf V} \alpha, \alpha \rangle
					\;  =  \;
	\langle  {\mathbf A}(\omega_\ast) {\mathbf V} \alpha, {\mathbf V} \alpha \rangle
					\;  \leq  \;
	\lambda_1  \left( \omega_\ast \right).	
\end{equation}
Thus we deduce 
\[
	\lambda_1 \left(  \omega_\ast \right)
						\;\; =	\;\;
	\max_{\omega \in \Omega} \; \lambda_1 \left(  \omega \right)
						\;\; = \;\;
	\langle  {\mathbf A}(\omega_\ast) {\mathbf V} \alpha, {\mathbf V} \alpha \rangle.
\]
Consequently, ${\mathbf V}\alpha$ is a unit eigenvector of ${\mathbf A}(\omega_\ast)$ corresponding to its largest eigenvalue,
where 
$
	\omega_\ast
$
belongs to
$
	\arg \max_{\omega \in \Omega} \; \lambda_1 \left( \omega \right).
$
Furthermore, ${\rm span} \left\{ {\mathbf V} \alpha \right\} \; \subseteq \; {\mathcal V}$. This proves assertion (ii).
\end{proof}

When ${\mathcal V}$ does not contain the optimal subspace ${\mathcal S}_\ast$ (as in part (ii) of Lemma \ref{thm:low_rank}), 
the next result quantifies the gap between the eigenvalues of the original and the reduced operators in 
terms of the distance from ${\mathcal S}_\ast$ to the $J$ dimensional subspaces of ${\mathcal V}$. For this result, we define 
the distance between two finite dimensional subspaces $\widetilde{\mathcal S}, {\mathcal S}$ of $\ell^2({\mathbb N})$
of same dimension by
\[
	d\left(  \widetilde{\mathcal S}, {\mathcal S}  \right)
				\;\;	:=	\;\;
	\max_{\widetilde{v} \in \widetilde{\mathcal S}, \; \| \widetilde{v} \|_2 = 1} \;\;		\min_{v \in {\mathcal S}}	\;\;
			\| \widetilde{v} - v \|_2.
\] 
This distance corresponds to the sine of the largest angle between the subspaces $\widetilde{\mathcal S}$ and ${\mathcal S}$.
Results of similar nature can be found in the literature, see for instance \cite[Theorem 11.7.1]{Parlett1998}
and \cite[Proposition 4.5]{Saad2011} where the bounds are in terms of distances between one dimensional
subspaces.

\begin{theorem}[Accuracy of Reduced Problems]\label{thm:accuracy_rproblems}
Let ${\mathcal V}$ be a subspace of $\ell^2({\mathbb N})$ with dimension $J$ or larger.
\begin{enumerate}

\item[\bf (i)] 
For each $\omega$ in $\overline{\Omega}$, we have 
	\begin{equation}\label{eq:rate_of_conv}
		\lambda_J \left( \omega \right) 	\; - \; \lambda^{\mathcal V}_J \left( \omega \right) \;\; = \;\; O \left( \varepsilon^2 \right),
	\end{equation}
	where 
	\[
		\varepsilon \; 	:=	\; 
						\min
						\left\{
							d\left(  \widetilde{\mathcal S}, {\mathcal S}  \right)
								\;\;	|	\;\;
						\widetilde{\mathcal S} \text{ is a J dimensional subspace of } {\mathcal V}	
						\right\}
	\]
	and  ${\mathcal S}$ is the subspace spanned by the eigenvectors corresponding to the $J$ largest eigenvalues of ${\mathbf A}(\omega)$. 

\item[\bf (ii)] The equality
	\begin{equation}\label{eq:rate_of_conv2}
		\max_{\omega \in \Omega} \; \lambda_J \left( \omega \right)  \;		-		\;
			\max_{\omega \in \Omega} \; \lambda^{\mathcal V}_J \left( \omega \right)  	\;\;	=	\;\; 	O \left( \varepsilon_\ast^2 \right)
	\end{equation}
	holds. Here, $\varepsilon_\ast$ is given by
	\[
		\varepsilon_\ast \; 	:=	\; 
						\min
						\left\{
							d\left(  \widetilde{\mathcal S}, {\mathcal S}_\ast  \right)
								\;\;	|	\;\;
						\widetilde{\mathcal S} \text{ is a J dimensional subspace of } {\mathcal V}	
						\right\}
	\]
	for some subspace ${\mathcal S}_\ast$ spanned by the eigenvectors corresponding to the $J$ largest eigenvalues of ${\mathbf A}(\omega_\ast)$
	at some $\omega_\ast \in \arg \max_{\omega \in \Omega} \; \lambda_J \left( \omega \right) $.
\end{enumerate}
\end{theorem}

\begin{proof}
\textbf{(i)}
Let $\widehat{\mathcal S}$ be a $J$ dimensional subspace of ${\mathcal V}$ such that $\varepsilon = d\left(  \widehat{\mathcal S}, {\mathcal S} \right)$.
Furthermore, for a given unit vector $\widehat{v} \in \widehat{\mathcal S}$, let us use the notations
\[
	v\left( \widehat{v} \right)	\;	:=	\;
	\arg \min	\;
			\left\{	\| \widehat{v} - v \|_2 \;\; | \;\; 	v \in {\mathcal S}	\right\}
			\;\;\;\;	 {\rm and}	\;\;\;\;
	\delta\left( \widehat{v} \right)	\;	:=	\;
		\widehat{v}	-	v\left( \widehat{v} \right),
\] 
where $\| \delta\left( \widehat{v} \right) \|_2 = O(\varepsilon)$. Observe that the inner minimization problem in the
definition of $d\left(  \widehat{\mathcal S}, {\mathcal S} \right)$ is a least-squares problem, 
so the minimizer $v\left( \widehat{v} \right)$ of  $\| \widehat{v} - v \|_2$ over $v \in {\mathcal S}$ defined above is unique and
$\delta\left(  \widehat{v} \right) \bot \; {\mathcal S}$. Additionally, $\| v(\widehat{v}) \|_2^2 = 1 - O\left( \varepsilon^2 \right)$
due to the properties $\| \widehat{v} \|_2 = \| v(\widehat{v}) + \delta(\widehat{v}) \|_2 = 1$ and $v(\widehat{v}) \; \bot \; \delta(\widehat{v})$.

Now the maximin characterization (\ref{eq:maximin}) for $\lambda_J^{\mathcal V}(\omega)$ implies
\begin{eqnarray*}
\lambda^{\mathcal V}_J(\omega) \;\;  \geq  \;\;	\min_{\widehat{v} \in \widehat{\mathcal S}, \; \| \widehat{v} \|_2 = 1} \;\; \langle {\mathbf A}(\omega) \widehat{v}, \widehat{v} \rangle
		 &=& 
	\min_{\widehat{v} \in \widehat{\mathcal S}, \; \| \widehat{v} \|_2 = 1} \;\; 
		\langle {\mathbf A}(\omega) \left[ v\left( \widehat{v} \right) + \delta\left( \widehat{v} \right) \right], \left[ v\left( \widehat{v} \right) + \delta\left( \widehat{v} \right) \right] \rangle \\
		&\geq&
		\min_{\widehat{v} \in \widehat{\mathcal S}, \; \| \widehat{v} \|_2 = 1} \;\; \langle {\mathbf A}(\omega) v(\widehat{v}), v(\widehat{v}) \rangle	\\
		& &	\hskip 4ex
+ \min_{\widehat{v} \in \widehat{\mathcal S}, \; \| \widehat{v} \|_2 = 1} \;\; 2\Re \left( \langle {\mathbf A}(\omega) v\left( \widehat{v} \right), \delta\left( \widehat{v} \right) \rangle \right)  \\
		& &	\hskip 4ex
	+ \min_{\widehat{v} \in \widehat{\mathcal S}, \; \| \widehat{v} \|_2 = 1} \;\; \langle {\mathbf A}(\omega) \delta\left( \widehat{v} \right), \delta\left( \widehat{v} \right) \rangle \\
		& \geq &	\lambda_J \left(  \omega  \right) \min_{\widehat{v} \in \widehat{\mathcal S}, \; \| \widehat{v} \|_2 = 1} \| v(\widehat{v}) \|_2^2	
										\;	-	\;	O(\varepsilon^2) \\
		&=& \lambda_J \left(  \omega  \right) \; - \; O(\varepsilon^2).
\end{eqnarray*}
Above, on the third line, we note that $\langle {\mathbf A}(\omega) v\left( \widehat{v} \right), \delta\left( \widehat{v} \right) \rangle = 0$ due to 
the fact ${\mathbf A}(\omega) v\left( \widehat{v} \right) \in {\mathcal S}$ and $\delta\left(  \widehat{v} \right) \bot \; {\mathcal S}$,
and on the second to the last line, we employ $\| v(\widehat{v}) \|_2^2 = 1 - O\left( \varepsilon^2 \right)$.
The desired equality \eqref{eq:rate_of_conv} follows from $\lambda^{\mathcal V}_J \left( \omega \right)  \leq  \lambda_J \left( \omega \right)$
due to Lemma \ref{lemma:monotonicity}.

\noindent
\textbf{(ii)} This is an immediate corollary of (i). In particular,
\[
		O(\varepsilon^2_\ast)	\;\; = \;\;			\lambda_J\left( \omega_\ast \right)	\; - \;	 \lambda^{\mathcal V}_J \left( \omega_\ast \right)	
							\;\; \geq \;\;			
	\max_{\omega \in \Omega} \; \lambda_J \left( \omega \right)  \;		-		\; \max_{\omega \in \Omega} \; \lambda^{\mathcal V}_J (\omega)
\]
combined with  
$
	\max_{\omega \in \Omega} \; \lambda^{\mathcal V}_J (\omega)	\; \leq \;	
	\max_{\omega \in \Omega} \; \lambda_J (\omega)
$ 
(due to Lemma \ref{lemma:monotonicity}) yield (\ref{eq:rate_of_conv2}).
\end{proof}

\subsection{The Greedy Procedure}
The basic greedy procedure solves the reduced eigenvalue optimization problem (\ref{eq:reduced_problems})
for a given subspace ${\mathcal V}$. Denoting a global optimizer of the reduced problem with $\omega_\ast$,
the subspace is expanded with the addition of the eigenvectors corresponding to $\lambda_1(\omega_\ast), \dots, \lambda_J(\omega_\ast)$,
then this is repeated with the expanded subspace. A formal description is given in Algorithm \ref{alg},
where ${\mathcal S}_k$ denotes the subspace at the $k$th step of the procedure, and 
$\lambda^{(k)}_J(\omega) := \lambda^{{\mathcal S}_k}_J(\omega)$. The reduced eigenvalue optimization
problems on line 5 are nonsmooth and nonconvex. The description assumes that these problems can be
solved globally. The algorithm in \cite{Mengi2014} works well in practice for this purpose when 
the number of parameters, $d$, is small.
These reduced problems are computationally cheap to solve, the main computational burden comes
from line 6 which requires the computation of eigenvectors of the full problem. In the finite dimensional
case, these large eigenvalue problems are typically solved by means of an iterative method, 
for instance by Lanczos' method.

\begin{algorithm}
 \begin{algorithmic}[1]
  \REQUIRE{ A parameter dependent compact self-adjoint operator ${\mathbf A}(\omega)$ of the form (\ref{eq:mat_func}), 
  and a compact subset $\;\; \Omega \subset \overline{\Omega} \;\;$ of ${\mathbb R}^d$.}
\STATE $\omega^{(1)} \gets$ a random point in $\Omega$.
\STATE $s^{(1)}_1, \dots, s^{(1)}_J \gets $ eigenvectors corresponding to $\lambda_1(\omega^{(1)}), \dots, \lambda_J(\omega^{(1)})$.
\STATE ${\mathcal S}_1	\;	\gets \; 
		{\rm span} \left\{  s^{(1)}_1, \dots, s^{(1)}_J  \right\}.
		$
	
\FOR{$k \; = \; 2, \; 3, \; \dots$}
	\STATE  
		$\omega^{(k)}$ \hskip -0.4ex $\gets$ \hskip -0.4ex 
		any $\omega_\ast \in \arg\min_{\omega \in \Omega} \lambda^{(k-1)}_J(\omega) \; $ for the minimization problem \textbf{(MN)}, or \\
			$\omega^{(k)}$ \hskip -0.4ex $\gets$ \hskip -0.4ex
			any $\omega_\ast \in \arg\max_{\omega \in \Omega} \lambda^{(k-1)}_J(\omega) \;$ for the maximization problem \textbf{(MX)}.
	\STATE $s^{(k)}_1, \dots, s^{(k)}_J  \gets $ eigenvectors corresponding to $\lambda_1(\omega^{(k)}), \dots, \lambda_J(\omega^{(k)})$.
	\STATE ${\mathcal S}_k \; \gets \; {\mathcal S}_{k-1} \oplus {\rm span} \left\{  s^{(k)}_1, \dots, s^{(k)}_J  \right\}$.
\ENDFOR
 \end{algorithmic}
\caption{The Greedy Subspace Procedure}
\label{alg}
\end{algorithm}

As numerical experiments demonstrate (see Section 5), the power of this greedy subspace procedure is that high accuracy 
is often reached after a small number of steps, for reduced problems of small size. This can mostly be
attributed to the following interpolatory properties between $\lambda^{(k)}_J(\omega)$ and $\lambda_J(\omega)$.

\begin{lemma}[Hermite Interpolation]\label{thm:first_der}
The following hold regarding Algorithm \ref{alg}:
\begin{enumerate}
	\item[\bf (i)] $\lambda_J\left( \omega^{(\ell)} \right)	\; = \;	\lambda^{(k)}_J\left( \omega^{(\ell)}	 \right)$
	for $\ell = 1,\dots,k$;
	\item[\bf (ii)] If $J > 1$, then $\lambda_{J-1}\left( \omega^{(\ell)} \right)	\; = \;	\lambda^{(k)}_{J-1}\left( \omega^{(\ell)}	 \right)$
	for $\ell = 1,\dots,k$;
	\item[\bf (iii)] If $\lambda_J \left( \omega^{(\ell)} \right)$ is simple, then
	$\lambda^{(k)}_J \left( \omega^{(\ell)} \right)$ is also simple for $\ell = 1,\; 2,\; \dots, k$;
	\item[\bf (iv)] If  $\lambda_J \left( \omega^{(\ell)} \right)$ is simple, then
	$\; \nabla \lambda_J \left( \omega^{(\ell)} \right)	\; = \;	\nabla \lambda^{(k)}_J \left( \omega^{(\ell)} \right) \;\;$
	for $\ell = 1,\; 2,\; \dots,k$.
\end{enumerate}
\end{lemma}
\begin{proof} 
\textbf{(i-ii)} Lines 2, 3, 6 and 7 of Algorithm \ref{alg} imply that $s^{(\ell)}_1, \dots, s^{(\ell)}_J \in {\mathcal S}_k$
for $\ell = 1, \dots, k$. An application of part (i) of Lemma \ref{lemma:interpolatary} with ${\mathcal V} = {\mathcal S}_k$ yields
$\lambda_j\left( \omega^{(\ell)} \right)  =  \lambda^{{\mathcal S}_k}_j \left( \omega^{(\ell)} \right)  =	\lambda^{(k)}_j \left( \omega^{(\ell)} \right)$
for $\ell = 1, \dots, k$ and $j = J$, as well as $j = J-1$ if $J > 1$, as desired.

\noindent
\textbf{(iii)} Suppose $\lambda^{(k)}_J \left(   \omega^{(\ell)}   \right)$ is not simple for some $\ell \in \{ 1, 2, \dots, k \}$. In this case
there must exist two mutually orthogonal unit eigenvectors $\hat{\alpha}, \tilde{\alpha} \in {\mathbb C}^m$ corresponding to it. 
Let us denote by ${\mathbf S}_k$ the operator as in (\ref{eq:change_coor}) in terms of a basis $S_k$ for ${\mathcal S}_k$,
and ${\mathbf A}^{{\mathcal S}_k} \left( \omega^{(\ell)} \right) = {\mathbf S}_k^\ast {\mathbf A} \left( \omega^{(\ell)} \right) {\mathbf S}_k$.
It follows from part (i) that
\begin{eqnarray*}
	\lambda_J \left(  \omega^{(\ell)} \right)	\; = \;
	\lambda^{(k)}_J \left(  \omega^{(\ell)}  \right)	 = 	
	\left\langle {\mathbf A}	\left(	\omega^{(\ell)} \right)	{\mathbf S}_k \hat{\alpha}, {\mathbf S}_k \hat{\alpha} \right\rangle  = 
	\min_{\alpha \in \widehat{\mathcal S}, \| \alpha \|_2 = 1}	\; 	\left\langle   {\mathbf A} \left( \omega^{(\ell)} \right)  {\mathbf S}_k \alpha, {\mathbf S}_k \alpha \right\rangle	\\
	 =  \left\langle	{\mathbf A}	\left(	\omega^{(\ell)} \right) {\mathbf S}_k \tilde{\alpha},	{\mathbf S}_k \tilde{\alpha} \right\rangle  = 
	\min_{\alpha \in \widetilde{\mathcal S}, \| \alpha \|_2 = 1}	\; 	\left\langle   {\mathbf A} \left( \omega^{(\ell)} \right)  {\mathbf S}_k \alpha, {\mathbf S}_k \alpha \right\rangle
\end{eqnarray*}
for some $J$ dimensional subspaces $\widehat{\mathcal S}, \widetilde{\mathcal S}$ of $\ell^2({\mathbb N})$ such that
$\hat{\alpha} \in \widehat{\mathcal S}$, $\tilde{\alpha} \in \widetilde{\mathcal S}$.
This shows that ${\mathbf S}_k \hat{\alpha}, \: {\mathbf S}_k \tilde{\alpha} \in \ell^2({\mathbb N})$ are mutually orthogonal eigenvectors 
corresponding to $\lambda_J 	\left(	\omega^{(\ell)} \right)$, so $\lambda_J\left( \omega^{(\ell)} \right)$ is not simple either.

\noindent
\textbf{(iv)} It follows from part (iii) that $\lambda^{(k)}_J\left( \omega^{(\ell)} \right)$ is also simple, so
both $\lambda_J\left( \omega \right)$ and $\lambda^{(k)}_J \left( \omega \right)$ 
are differentiable at $\omega^{(\ell)}$, furthermore the associated 
unit eigenvectors can be chosen in a way so that they are also differentiable at $\omega^{(\ell)}$ 
(see \cite[pages 57-58, Theorem 1]{Rellich1969} for the differentiability of $\lambda_J(\omega)$ and the associated unit eigenvector,
and \cite[pages 33-34, Theorem 1]{Rellich1969} for the differentiability of $\lambda_J^{(k)}(\omega)$ and the associated unit eigenvector).
By Lemma \ref{lemma:interpolatary} part (ii), the eigenvector 
	$\alpha_J$ of $\; {\mathbf A}^{{\mathcal S}_k} \left(	  \omega^{(\ell)} 	\right) = {\mathbf S}_k^\ast {\mathbf A} \left(	  \omega^{(\ell)} 	\right) {\mathbf S}_k \;$ 
corresponding to the eigenvalue $\lambda^{(k)}_J \left(	\omega^{(\ell)}	\right) = \lambda^{{\mathcal S}_k}_J \left( \omega^{(\ell)} \right)$ 
satisfies  $\alpha_J = {\mathbf S}_k^\ast s^{(\ell)}_J$. Equivalently, we have $s^{(\ell)}_J = {\mathbf S}_k \alpha_J$ 
(since $s^{(\ell)}_J \in {\mathcal S}_k$). By employing the analytical formulas for the derivatives of eigenvalue 
functions (see \cite{Lancaster1964} for the finite dimensional matrix-valued case whose derivation exploits the Hermiticity
of the matrix-valued function; the generalization to the infinite dimensional case, leading to the 
formula $\partial \lambda_J(\omega)/\partial \omega_q = \langle \partial {\mathbf A} (\omega) / \partial \omega_q s_J, s_J \rangle$
where $s_J$ is a unit eigenvector corresponding to $\lambda_J(\omega)$, is straightforward by making use of the self-adjointness 
of ${\mathbf A}(\omega)$), for $q = 1, \dots, d$ we obtain
\begin{eqnarray*}
	\frac{\partial \lambda^{(k)}_J \left( \omega^{(\ell)}  \right) } {\partial \omega_q}
				\;\; &  =  & \;\; 
	\left\langle	\frac{\partial {\mathbf A}^{{\mathcal S}_k} \left( \omega^{(\ell)} \right) }{\partial \omega_q}	\alpha_J,  \alpha_J 	\right\rangle	\\
				\;\; &  =  & \;\;
	\left\langle 	\frac{\partial {\mathbf A} \left( \omega^{(\ell)} \right) }{\partial \omega_q}	{\mathbf S}_k \alpha_J,  {\mathbf S}_k \alpha_j	\right\rangle	\\
				\;\; & = & \;\;
	\left\langle		\frac{\partial {\mathbf A} \left( \omega^{(\ell)} \right) }{\partial \omega_q}	s^{(\ell)}_J,  s^{(\ell)}_J		\right\rangle
				\;\; = \;\;
	\frac{\partial \lambda_J \left(	\omega^{(\ell)} \right) } {\partial \omega_q}.
\end{eqnarray*}
This completes the proof.
\end{proof}

\subsection{The Extended Greedy Procedure}
To better exploit the Hermite interpolation properties of Lemma \ref{thm:first_der}, we extend the basic
greedy procedure of the previous subsection with the inclusion of additional eigenvectors in the
subspaces at points close to the optimizers of the reduced problems. The purpose here is to
achieve 
\[
	\lim_{k\rightarrow \infty} \:
	\left\| \nabla^2 \lambda_J ( \omega^{(k)} ) - \nabla^2 \lambda_J^{(k)} ( \omega^{(k)} ) \right\|_2 \; = \; 0
\]
in addition to $\lambda_J(\omega^{(k)}) = \lambda_J^{(k)} (\omega^{(k)})$ and 
$\nabla \lambda_J(\omega^{(k)}) = \nabla \lambda_J^{(k)} (\omega^{(k)})$. These properties enable us
to make an analogy with a quasi-Newton method for unconstrained smooth optimization, and come
up with a theoretical superlinear rate-of-convergence result in the next section.

In the extended procedure also the reduced eigenvalue optimization problem (\ref{eq:reduced_problems}) 
is solved for a subspace ${\mathcal V}$ already constructed. 
Denoting the optimizer of the reduced problem with $\omega_\ast$,
in addition to the eigenvectors corresponding to $\lambda_1(\omega_\ast), \dots, \lambda_{J}(\omega_\ast)$,
the eigenvectors corresponding to $\lambda_1(\omega_\ast + h e_{pq}), \dots, \lambda_J(\omega_\ast + h e_{pq})$
are added into the subspace ${\mathcal V}$ for $h$ decaying to zero if convergence occurs, for $p = 1, \dots, d$ and $q = p, \dots, d$.
Here $e_{p q} := (1/\sqrt{2}) ( e_p  + e_q  )$ for $p \neq q$ as well as $e_{pp} := e_p$. 
We provide a formal description of this extended subspace procedure in Algorithm \ref{alge} below.

\begin{algorithm}
 \begin{algorithmic}[1]
  \REQUIRE{ A parameter dependent compact self-adjoint operator ${\mathbf A}(\omega)$ of the form (\ref{eq:mat_func}), 
  and a compact subset $\;\; \Omega \subset \overline{\Omega} \;\;$ of ${\mathbb R}^d$.}
\STATE $\omega^{(1)} \gets$ a random point in $\Omega$.
\STATE $s^{(1)}_1, \dots, s^{(1)}_{J} \gets $ eigenvectors corresponding to $\lambda_1(\omega^{(1)}), \dots, \lambda_{J}(\omega^{(1)})$.
\STATE ${\mathcal S}_1	\;	\gets \; 
		{\rm span} \left\{  s^{(1)}_1, \dots, s^{(1)}_{J}  \right\}.
		$
	
\FOR{$k \; = \; 2, \; 3, \; \dots$}
	\STATE  
			$\omega^{(k)}$ \hskip -0.4ex $\gets$ \hskip -0.4ex any $\omega_\ast \in \arg\min_{\omega \in \Omega} \lambda^{(k-1)}_J(\omega) \;$ for the minimization problem \textbf{(MN)}, or \\
			$\omega^{(k)}$ \hskip -0.4ex $\gets$ \hskip -0.4ex any $\omega_\ast \in \arg\max_{\omega \in \Omega} \lambda^{(k-1)}_J(\omega) \;$ for the maximization problem \textbf{(MX)}.
	\STATE $s^{(k)}_1, \dots, s^{(k)}_{J}  \gets $ eigenvectors corresponding to $\lambda_1(\omega^{(k)}), \dots, \lambda_{J}(\omega^{(k)})$. 
	\STATE $h^{(k)} \gets \| \omega^{(k)} - \omega^{(k-1)} \|_2$	
	\FOR{$p \; = \; 1, \dots, d, \;\; q \; = \; p, \dots, d$}
	\STATE $s^{(k)}_{1,p q}, \dots, s^{(k)}_{J, p q}  \gets $ eigenvectors corresponding to \\
		\hskip 18ex									$\lambda_1(\omega^{(k)} + h^{(k)} e_{pq}), \dots, \lambda_J(\omega^{(k)} + h^{(k)} e_{pq})$.
	\ENDFOR
	\STATE ${\mathcal S}_k \; \gets \; {\mathcal S}_{k-1} \oplus {\rm span} \left\{  s^{(k)}_1, \dots, s^{(k)}_{J}  \right\} 
				\oplus \left\{ \bigoplus_{p = 1, q = p}^d {\rm span} \left\{  s^{(k)}_{1,p q}, \dots, s^{(k)}_{J, p q}  \right\} \right\}$.
\ENDFOR
 \end{algorithmic}
\caption{The Extended Greedy Subspace Procedure}
\label{alge}
\end{algorithm}

The reduced eigenvalue functions of the extended procedure possesses additional interpolatory properties
stated in the lemmas below. Their proofs are similar to the proofs for Lemma \ref{thm:first_der}, so we omit them. 
\begin{lemma}[Extended Hermite Interpolation]\label{thm:first_der_extended}
The assertions of Lemma \ref{thm:first_der} hold for Algorithm \ref{alge}. Additionally, we have
\begin{enumerate}
	\item[\bf (i)]
	$\lambda_J\left( \omega^{(\ell)} + h^{(\ell)} e_{pq} \right)	\; = \;	\lambda^{(k)}_J\left( \omega^{(\ell)} + h^{(\ell)} e_{pq}	 \right) \;\;$,
	\item[\bf (ii)] If $\lambda_J \left( \omega^{(\ell)} + h^{(\ell)} e_{pq} \right)$ is simple, then
	$\lambda^{(k)}_J \left( \omega^{(\ell)} + h^{(\ell)} e_{pq} \right)$ is also simple,
	\item[\bf (iii)] If  $\lambda_J \left( \omega^{(\ell)} + h^{(\ell)} e_{pq} \right)$ is simple,
	$\; \nabla \lambda_J  \left( \omega^{(\ell)} + h^{(\ell)} e_{pq} \right)	=	\nabla \lambda^{(k)}_J \left( \omega^{(\ell)} + h^{(\ell)} e_{pq} \right) \;\;$
\end{enumerate}
for every $\ell = 1,\dots,k$, $\: p = 1,\dots,d$ and $q = p,\dots,d$. 
\end{lemma}

The main motivation for the inclusion of additional eigenvectors in the subspaces is the deduction of theoretical
bounds on the proximity of the second derivatives, which we present next. This result is initially established
under the assumption that the third derivatives of the reduced eigenvalue functions $\omega\mapsto \lambda_J^{(k)}(\omega)$ 
are bounded uniformly with respect to $k$ provided $k$ is large enough. Subsequently, we show in Proposition \ref{prop:uniformb_3rdd}
that this assumption is always satisfied, hence can be dropped. 
\begin{lemma}\label{thm:sec_der_ext}
Suppose that the sequence $\{ \omega^{(k)} \}$ by Algorithm~\ref{alge} (or Algorithm~\ref{alg} when $d=1$ by
defining $h^{(k)} := | \omega^{(k)} - \omega^{(k-1)} |$) is convergent, that its limit
$\omega_\ast := \lim_{k\rightarrow \infty} \omega^{(k)}$ lies strictly in the interior of $\Omega$, and that 
$\lambda_J(\omega_\ast)$ is simple. Assume furthermore that in an open ball containing $\omega_*$, 
all third derivatives of functions $\omega\mapsto \lambda_J^{(k)}(\omega)$ are bounded 
uniformly with respect to $k\geq k_0$, with $k_0$ sufficiently large.
 Then the following assertions hold:
\begin{enumerate}
\item[\bf (i)]
There exists a constant $C > 0$ such that
\begin{equation}\label{eq:accuracy_sec_der}
	\left|
		\frac{\partial^2 \lambda_J \left( \omega^{(k)} \right)}{\partial \omega_p \partial \omega_q}  -  \frac{\partial^2 \lambda^{(k)}_J \left( \omega^{(k)} \right)}
			{\partial \omega_p \partial\omega_q}
	 \right|_2
	 	\; 	\leq		\;
	C h^{(k)}	\quad {\rm for} \; p,q = 1,\dots,d,
\end{equation}
in particular
\[
	\left\|
	\nabla^2 \lambda_J \left( \omega^{(k)} \right)
				-
	\nabla^2 \lambda^{(k)}_J \left( \omega^{(k)} \right)
	 \right\|_2
	 	\; 	\leq		\;
	d C h^{(k)},
\]
for all $k$ large enough;
\item[\bf (ii)] Additionally, if $\nabla^2 \lambda_J(\omega_\ast)$ is invertible, then
	\begin{equation}\label{eq:inv_Hessian}
		\left\|   \left[ \nabla^2 \lambda_J ( \omega^{(k)} )	\right]^{-1} 	-   
				\left[ \nabla^2 \lambda^{(k)}_J ( \omega^{(k)} )  \right]^{-1} 		\right\|_2
				\;	=	\;
			O( h^{(k)} )
	\end{equation}
for all $k$ large enough.
\end{enumerate}
\end{lemma}
\begin{proof}
\textbf{(i)} First we specify a ball centered at $\omega_\ast$ in which $\lambda_J(\omega)$ and
$\lambda_J^{(k)}(\omega)$ for all large $k$ are simple. The argument initially assumes $J > 1$.
Letting $\varepsilon := \min \{ \lambda_{J-1}(\omega_\ast) - \lambda_J(\omega_\ast), \lambda_J(\omega_\ast) - \lambda_{J+1}(\omega_\ast) \}$,
consider the ball ${\mathcal B}(\omega_\ast, \varepsilon/(8\gamma))$, where $\gamma$ is the Lipschitz constant as
in Lemma \ref{lemma:Lipschitz_continuity}. 
Without loss of generality, let us assume ${\mathcal B}(\omega_\ast, \varepsilon/(8\gamma)) \subseteq \Omega$
(i.e., otherwise choose $\varepsilon$ even smaller so that ${\mathcal B}(\omega_\ast, \varepsilon/(8\gamma)) \subseteq \Omega$).
Now, due to $\lambda_{J-1}(\omega_\ast) - \lambda_J(\omega_\ast) \geq \varepsilon$ as well as
$\lambda_J(\omega_\ast) - \lambda_{J+1}(\omega_\ast) \geq \varepsilon$, and
by Lemma \ref{lemma:Lipschitz_continuity}, we have 
\begin{equation}\label{eq:gap}
	\lambda_{J-1}(\omega) - \lambda_{J}(\omega)	\geq	3 \varepsilon / 4
		\;\;	{\rm and}		\;\;
	\lambda_J(\omega) - \lambda_{J+1}(\omega) 	\geq	3 \varepsilon / 4
	\quad	\forall \omega \in {\mathcal B}(\omega_\ast, \varepsilon/(8\gamma)). 
\end{equation}

Next choose $k$ large enough so that ${\mathcal B}(\omega^{(k)}, h^{(k)}) \subset {\mathcal B}(\omega_\ast, \varepsilon/(8\gamma))$. 
We will show that $\lambda^{(k)}_J(\omega)$ is also simple in ${\mathcal B}(\omega_\ast, \varepsilon/(8\gamma))$. In this respect we note that 
$\lambda_j(\omega^{(k)}) = \lambda_j^{(k)}(\omega^{(k)})$ for $j = J-1, J$
due to parts (i) and (ii) of Lemma \ref{thm:first_der},
so from (\ref{eq:gap}) we have
\[
	\begin{split}
	\lambda_{J-1}^{(k)}(\omega^{(k)}) - \lambda_{J}^{(k)}(\omega^{(k)})	\;\; \geq \;\;	3 \varepsilon / 4
		\quad	{\rm and}	\hskip 24ex	\\
	\lambda_J^{(k)}(\omega^{(k)}) - \lambda_{J+1}^{(k)}(\omega^{(k)}) 	\;\; \geq \;\;		\lambda_J^{(k)}(\omega^{(k)}) - \lambda_{J+1}(\omega^{(k)})	
				\;\; \geq \;\;	3 \varepsilon / 4,
	\end{split}
\]
where in the last line we also used $\lambda_{J+1}^{(k)}(\omega^{(k)}) \leq \lambda_{J+1}(\omega^{(k)})$
which holds due to monotonicity (Lemma \ref{lemma:monotonicity}).
Another application of Lemma \ref{lemma:Lipschitz_continuity} yields
\[
	\lambda_{J-1}^{(k)}(\omega) - \lambda_{J}^{(k)}(\omega)	\geq	 \varepsilon / 4
		\quad	{\rm and}		\quad
	\lambda_J^{(k)}(\omega) - \lambda_{J+1}^{(k)}(\omega) 	\geq	 \varepsilon / 4
	\quad	\forall \omega \in {\mathcal B}(\omega_\ast, \varepsilon/(8\gamma)).
\]
We remark that the argument above applies to the case $J = 1$ trivially by considering
only the gaps  $\lambda_J(\omega) - \lambda_{J+1}(\omega)$ as well as
$\lambda_J^{(k)}(\omega) - \lambda_{J+1}^{(k)}(\omega)$ and showing they remain
bounded away from zero for all $\omega$ inside ${\mathcal B}(\omega_\ast, \varepsilon/(8\gamma))$. 

We prove the desired bounds (\ref{eq:accuracy_sec_der}) first for the sequence $\{ \omega^{(k)} \}$
generated by Algorithm \ref{alge}.
The simplicity of the eigenvalue functions $\lambda_J(\omega)$ and $\lambda_J^{(k)}(\omega)$ 
on ${\mathcal B}(\omega_\ast, \varepsilon/(8\gamma))$ imply that,  for each $p, q$, the functions 
\begin{equation}\label{eq:linesearch_fun}
	\ell_{pq}(t) := \lambda_J(\omega^{(k)} + t h^{(k)} e_{pq})
	\quad	{\rm and}	\quad 
	\ell_{pq,k}(t) := \lambda^{(k)}_J(\omega^{(k)} + t h^{(k)} e_{pq})
\end{equation}
are analytic on $(0,1)$. By applying Taylor's theorem to $\ell_{pq}(t), \ell_{pq,k}(t)$
on the interval $(0,1)$, we obtain
\begin{align*}
	\ell_{pq}(1)
				\; & = \;
	\ell_{pq}(0)
			+
	\ell'_{pq}(0)
			+
	\ell^{''}_{pq}(0) / 2
			+
	\ell^{'''}_{pq}(\eta) / 6	\\
	\ell_{pq,k}(1)
				\; & = \;
	\ell_{pq,k}(0)
			+
	\ell'_{pq,k}(0)
			+
	\ell''_{pq,k}(0) / 2
			+
	\ell'''_{pq,k}(\eta^{(k)}) / 6
\end{align*}
for some $\eta, \eta^{(k)} \in (0,1)$. 
Now by employing $\ell_{pq}(0) = \ell_{pq,k}(0)$, $\ell'_{pq}(0) = \ell'_{pq,k}(0)$ due to parts (i), (iv) of Lemma \ref{thm:first_der},
and $\ell_{pq}(1) = \ell^{(k)}_{pq}(1)$ due to part (i) of Lemma \ref{thm:first_der_extended}, we deduce
\[
	\frac{ \ell''_{pq}(0)  -  \ell''_{pq,k}(0) }{2}
				\; = \;
	\frac{ \ell'''_{pq,k} (\eta^{(k)})  -  \ell^{'''}_{pq}(\eta) }{6}.
\]
In the last expression, 
\[
	\ell''_{pq}(0) = \left[ h^{(k)} \right]^2 e_{pq}^T \nabla^2 \lambda_J(\omega^{(k)}) e_{pq}
	\quad
	{\rm and}
	\quad
	\ell''_{pq,k}(0) = \left[ h^{(k)} \right]^2 e_{pq}^T \nabla^2 \lambda_J^{(k)}(\omega^{(k)}) e_{pq}
\]
as well as $\; \ell'''_{pq,k} (\eta^{(k)}) = O \left( \left[ h^{(k)} \right]^3 \right)$ and $\ell^{'''}_{pq}(\eta) = O\left( \left[ h^{(k)} \right]^3 \right)$,
so we have
\begin{equation}\label{eq:Hessian_identity}
	\frac
	{
	\left[ h^{(k)} \right]^2
	\left|
	e_{pq}^T
	\left[
	\nabla^2 \lambda_J \left( \omega^{(k)} \right)
				-
	\nabla^2 \lambda^{(k)}_J \left( \omega^{(k)} \right)
	\right]
	e_{pq}
	\right|
	}{2}
			\leq
	\frac{D}{6} \left[ h^{(k)} \right]^3
\end{equation}
for some constant $D$ independent of $k$.

Choosing $q = p$ in (\ref{eq:Hessian_identity}) yields
\[
	\left|
	\frac{\partial^2 \lambda_J \left( \omega^{(k)} \right)}{\partial \omega_p^2}  -  \frac{\partial^2 \lambda^{(k)}_J \left( \omega^{(k)} \right)}{\partial \omega_p^2}
	\right|
			\leq
	\frac{D}{3} h^{(k)}.
\]
Additionally, for $q \neq p$, rewriting (\ref{eq:Hessian_identity}) as
\[
	\left|
	\frac{1}{2}
	\left(
	\sum_{\ell = p, q}
		\frac{\partial^2 \lambda_J \left( \omega^{(k)} \right)}{\partial \omega_\ell^2}  -  \frac{\partial^2 \lambda^{(k)}_J \left( \omega^{(k)} \right)}{\partial \omega_\ell^2}
	\right)	
			+
	\left(
	\frac{\partial^2 \lambda_J \left( \omega^{(k)} \right)}{\partial \omega_p \partial \omega_q}  -  \frac{\partial^2 \lambda^{(k)}_J \left( \omega^{(k)} \right)}{\partial \omega_p \partial \omega_q}
	\right)
	\right|
			\leq
	\frac{D}{3} h^{(k)},
\]
we obtain
\[
	\left|
\frac{\partial^2 \lambda_J \left( \omega^{(k)} \right)}{\partial \omega_p \partial \omega_q}  -  \frac{\partial^2 \lambda^{(k)}_J \left( \omega^{(k)} \right)}{\partial \omega_p \partial\omega_q}
	\right|	
			\leq
	\frac{2D}{3} h^{(k)}
\]
leading us to (\ref{eq:accuracy_sec_der}). 
The proof above establishing the bounds (\ref{eq:accuracy_sec_der}) also applies
to the sequence $\{ \omega^{(k)} \}$ generated by Algorithm \ref{alg} when $d = 1$ by defining
$\widetilde{h}^{(k)} := \omega^{(k-1)} - \omega^{(k)}$ (so that $h^{(k)} = | \widetilde{h}^{(k)} |$)
and letting
\[
	\ell_{pq}(t) := \lambda_J(\omega^{(k)} + t \widetilde{h}^{(k)} e_{pq})
	\quad	{\rm and}	\quad 
	\ell_{pq,k}(t) := \lambda^{(k)}_J(\omega^{(k)} + t \widetilde{h}^{(k)} e_{pq})
\]
instead of (\ref{eq:linesearch_fun}).

\textbf{(ii)} 
From part (i), as well as by the
existence of $\nabla^2 \lambda_J(\omega^{(k)})$, $\nabla^2 \lambda_J^{(k)}(\omega^{(k)})$
for all $k$ large enough so that ${\mathcal B}(\omega^{(k)}, h^{(k)}) \subset {\mathcal B}(\omega_\ast, \varepsilon/(8\gamma))$
and by the continuity of $\nabla^2 \lambda_J(\omega)$ at $\omega = \omega_\ast$, we have
\[
	\lim_{k\rightarrow \infty} \nabla^2 \lambda_J(\omega^{(k)}) \;\; = \;\; \lim_{k\rightarrow \infty} \nabla^2 \lambda_J^{(k)}(\omega^{(k)}) \;\; = \;\; \nabla^2 \lambda_J(\omega_\ast).
\]
Exploiting this and the invertibility of $\nabla^2 \lambda_J(\omega_\ast)$, we deduce that 
$\nabla^2 \lambda_J(\omega^{(k)})$  and $\nabla^2 \lambda_J^{(k)}(\omega^{(k)})$ are invertible
for all $k$ large enough. 

For such a $k$, letting $A =  \nabla^2 \lambda_J ( \omega^{(k)} ), \widetilde{A} = \nabla^2 \lambda^{(k)}_J ( \omega^{(k)} )$
as well as $X = [ \nabla^2 \lambda_J ( \omega^{(k)} )	  ]^{-1}$, $\widetilde{X} = [ \nabla^2 \lambda^{(k)}_J ( \omega^{(k)} ) ]^{-1}$,
the classical adjugate formula (i.e., $[B^{-1}]_{ij} = \frac{(-1)^{i+j} {\rm det} (B_{ji})}{ {\rm det} (B)}$
for an invertible $B$, where $B_{ji}$ denotes the matrix obtained from $B$ by removing its $j$th row and the $i$th column)
yields
\[
	\widetilde{x}_{ij}
		=
	\frac{(-1)^{(i+j)}\det(\widetilde{A}_{ji})}{\det(\widetilde{A})}
	\quad {\rm for} \; i,j = 1,\dots,d
\]
By part (i), in particular equation (\ref{eq:accuracy_sec_der}),
we have $\det(\widetilde{A}_{ji}) = \det(A_{ji}) + O(h^{(k)})$ and $\det(\widetilde{A}) = \det(A) + O(h^{(k)})$, so we deduce
\[
	\widetilde{x}_{ij}
		=
	\frac{(-1)^{i+j} \det(A_{ji}) + O(h^{(k)})}{\det(A) + O(h^{(k)})}
		=
	\frac{(-1)^{i+j} \det(A_{ji})}{\det(A)}	+	O(h^{(k)})
		=
		x_{ij} + O(h^{(k)})
\]
leading to (\ref{eq:inv_Hessian}).
\end{proof}

\smallskip

\begin{proposition}\label{prop:uniformb_3rdd}
Suppose that the sequence $\{ \omega^{(k)} \}$ by Algorithm~\ref{alge} is convergent, its limit
$\omega_\ast := \lim_{k\rightarrow \infty} \omega^{(k)}$ lies strictly in the interior of $\Omega$ and
that $\lambda_J(\omega_\ast)$ is simple. Then for $k$ sufficiently large all third derivatives of functions 
$\omega\mapsto \lambda_J^{(k)}(\omega)$ exist in an open interval containing $\omega_*$,  and they 
can be bounded uniformly with respect to $k$.
\end{proposition}
\begin{proof}
For the sake of clarity we first consider the case when $d=1$.
Following from the assumptions of the proposition and from the elements spelled out in the proof of Lemma~\ref{thm:sec_der_ext}, there exists an open interval $\mathcal{I}:=(\omega_*-r_1,\ \omega_*+r_1)$, and numbers $k_1\in\mathbb{N}$ and $\ell \in\mathbb{R}, \ell>0$, such that for every $k\geq k_1$ and all $\omega\in\mathcal{I}$, the eigenvalue $\lambda_J^{(k)}(\omega)$ is simple, as well as
\begin{equation}\label{fr1}
\min\left\{ \left|\lambda_J^{(k)}(\omega)-\lambda_{J-1}^{(k)}(\omega)\right|,\
\left|\lambda_J^{(k)}(\omega)-\lambda_{J+1}^{(k)}(\omega)\right|\right\}
  \geq \ell.
\end{equation}

Since functions $f_{\ell}$ are assumed real analytic, there exist  analytic extensions on the complex plane at $\omega_*$, which we refer to by functions $\hat f_{\ell},\ \ell=1,\ldots \kappa$. We denote by  $\mathbf{\hat A}^{\mathcal{S}_k}(\omega)$ the corresponding extension of  $\mathbf{A}^{\mathcal{S}_k}(\omega)$. Following the ideas spelled out in the proof of Lemma~3 of \cite{Kressner2017}, we can express
\begin{equation}\label{fr2}
\begin{array}{ll}
\left\|\mathbf{\hat A}^{\mathcal{S}_k}(\omega)-\mathbf{\hat A}^{\mathcal{S}_k}(\omega_*)\right\|_2
&\leq \sum_{\ell=1}^{\kappa} \left\|\mathbf{S}_k^*\mathbf{A}_{\ell} \mathbf{S}_k\right\|_2 \left| \hat f_\ell({\omega})- \hat{f}_\ell(\omega_*)\right|
\\
&\leq \sum_{\ell=1}^{\kappa} \left\|\mathbf{A}_{\ell} \right\|_2 \left| \hat f_\ell({\omega})- \hat{f}_\ell(\omega_*)\right|.
\end{array}
\end{equation}
Note that if $\omega$ is non-real, the difference $\mathbf{\hat A}^{\mathcal{S}_k}(\omega)-\mathbf{\hat A}^{\mathcal{S}_k}(\omega_*)$ might be non-Hermitian. 
However, an important conclusion from the above inequalities is that 
$\| \mathbf{\hat A}^{\mathcal{S}_k}(\omega)-\mathbf{\hat A}^{\mathcal{S}_k}(\omega_*) \|_2$ can be bounded from above uniformly independent of $k$,
that is independent of the dimension of the subspace ${\mathcal S}_k$, for all $\omega$ inside a disk centered at $\omega_\ast$.
Furthermore, this upper bound can be chosen arbitrarily close to 0 by reducing the radius of the disk.

From (\ref{fr1}), (\ref{fr2}) and Theorem~5.1 of \cite{Stewart1990}, we conclude that there exists a number $r_2\in (0,\ r_1)$ independent of $k\geq k_1$ (i.e., the dimension of the subspace), such that the eigenvalue function $\hat \lambda_J^{(k)}$, obtained by replacing $\mathbf{ A}^{\mathcal{S}_k}$ with its analytic exension $\mathbf{\hat A}^{\mathcal{S}_k}$, remains simple on and inside the disk $\left\{\omega\in\mathbb{C}:\ |\omega-\omega_*|\leq r_2 \right\}$
for every $k\geq k_1$.
Hence $\hat \lambda_J^{(k)}$ is well-defined and analytic inside this disk 
for every $k\geq k_1$.
 For $\tilde\omega\in\mathbb{R}$ satisfying $|\tilde\omega-\omega_*|< r_2/2$, we have
\[
\frac{d^3\hat \lambda_J^{(k)}}{d\omega^3}(\tilde\omega)=\frac{3!}{2\pi {\rm i}} \oint_{|\omega-\tilde\omega|=r_2/2}
\frac{\hat\lambda_J^{(k)}(\omega)}{(\omega-\tilde\omega)^4}d\omega, 
\]
with ${\rm i}$ the imaginary unit. We can bound
\begin{equation}\label{fr3}
\begin{array}{lll}
\left|\frac{ d^3\hat\lambda_J^{(k)}  } {d\omega^3}(\tilde\omega)\right|
&\leq &\frac{96}{r_2^3} \ \max_{\omega\in\mathbb{C},|\omega-\tilde\omega|=r/2} \left|\hat\lambda_J^{(k)}(\omega)\right|
\\
&\leq & \frac{96}{r_2^3}\	
\sum_{\ell = 1}^\kappa \left\|{\mathbf A}^{{\mathcal S}_k}_\ell\right\|_2	 \left(\max_{\omega\in\mathbb{C},\ |\omega-\tilde\omega|=r_2/2}	\left|\hat f_\ell(\omega)\right|\right) \\
&\leq & \frac{96}{r_2^3}\	
\sum_{\ell = 1}^\kappa \left\|{\mathbf A}_\ell\right\|_2	 \left(\max_{\omega\in\mathbb{C},\ |\omega-\tilde\omega|=r_2/2}	\left|\hat f_\ell(\omega)\right|\right) ,
\end{array}
\end{equation}
hence, an upper bound is established that does not depend on $k\geq k_1$.

\medskip 
For the case $d>1$ only a slight adaptation is needed for the mixed derivatives. We sketch the main principles by means of the case when $d=2$. 
From the ideas behind (\ref{fr3}) a bound on $\hat\lambda_J^{(k)}$ induces a bound, uniform in $k$, on its partial derivative 
$\frac{\partial^2 \hat\lambda_J^{(k)}}{\partial \omega_1^2}$ in an open set containing $\omega_*$. Likewise the bound on the latter induces a bound 
on $\frac{\partial^3 \hat\lambda_J^{(k)}}{\partial \omega_1^2\partial\omega_2}$.
 \end{proof}

\section{Convergence Analysis}\label{sec:convergence}
In this section we analyze the convergence properties of the two subspace procedures introduced.
The first part concerns global convergence: it elaborates on whether the iterates of  
Algorithm \ref{alg} and Algorithm \ref{alge} necessarily converge as the subspace dimensions grow to infinity,
and if they do converge, where they converge to. The second part establishes a superlinear
rate-of-convergence result for the iterates of Algorithm \ref{alge}
(and Algorithm \ref{alg} when $d=1$).


\subsection{Global Convergence}\label{sec:global_convergence}
The subspace procedures, when applied to the minimization problem \textbf{(MN)}, converge globally. 
A formal statement of this global convergence property together with its proof are given in what follows. Note that the 
sequence $\{\omega^{(k)} \}$ generated by Algorithm \ref{alg} or Algorithm \ref{alge} belongs to the bounded set $\Omega$,
so it must have convergent subsequences.
\begin{theorem}\label{thm:global_conv}
The following hold for both Algorithm \ref{alg} and Algorithm \ref{alge}. 
\begin{enumerate}
	\item[\bf (i)] The limit of every convergent subsequence of the sequence $\left\{ \omega^{(k)} \right\}$ 
	for the minimization problem \textbf{(MN)} is a global minimizer of $\lambda_J(\omega)$ over $\Omega$.

	\item[\bf (ii)] 
	$
		\lim_{k\rightarrow \infty} \lambda_J^{(k)}(\omega^{(k+1)}) \;\; = \;\;  \lim_{k\rightarrow \infty} \: \min_{\omega \in \Omega} \lambda_J^{(k)}(\omega)
					\;\; = \;\; \min_{\omega \in \Omega} \lambda_J(\omega).
	$
\end{enumerate}
\end{theorem}
\begin{proof}
\textbf{(i)} Let $\{ \omega^{(\ell_k)} \}$ be a convergent subsequence of $\{ \omega^{(k)} \}$.
By the monotonicity property, that is by Lemma \ref{lemma:monotonicity}, we have
\begin{equation}\label{eq:ubound}
	\begin{split}
	\min_{\omega \in \Omega} \: \lambda_J(\omega)		\;\; & \geq \;\;		\min_{\omega \in \Omega} \: \lambda^{(\ell_{k+1}-1)}_J (\omega)	\\
										\;\; & = \;\;	\lambda^{(\ell_{k+1}-1)}_J ( \omega^{(\ell_{k+1})} )
										\;\; \geq	\;\;	\lambda^{(\ell_k)}_J ( \omega^{(\ell_{k+1})} ),
	\end{split}
\end{equation}
and, by the interpolatory property, that is part (i) of Lemma \ref{thm:first_der}, 
\begin{equation}\label{eq:lbound}
	\min_{\omega \in \Omega} \lambda_J(\omega)	\;\; \leq \;\;	\lambda_J(\omega^{(\ell_k)}) \;\; = \;\; \lambda^{(\ell_k)}_J (\omega^{(\ell_k)}).	\hskip 12ex
\end{equation}
But observe the Lipschitz continuity of $\lambda_J^{(\ell_k)}(\omega)$ (see Lemma \ref{lemma:Lipschitz_continuity}) implies
\[
	\lim_{k \rightarrow \infty}	\;
	 \left| \lambda^{(\ell_k)}_J ( \omega^{(\ell_{k+1})} )	-	\lambda^{(\ell_k)}_J ( \omega^{(\ell_k)} ) \right|
	 			\; = \;
	\gamma
	\lim_{k \rightarrow \infty}	\;
	 \| \omega^{(\ell_{k+1})}  -  \omega^{(\ell_k)}  \|_2 \; = \; 0.
\]
Consequently, taking the limits in (\ref{eq:ubound}) and (\ref{eq:lbound}) as $k \rightarrow \infty$ and employing
the interpolatory property (part (i) of Lemma \ref{thm:first_der}), we deduce
\begin{equation}\label{eq:all_converge}
	\lim_{k\rightarrow \infty} \lambda_J (\omega^{(\ell_k)})
			\;\; = \;\;
	\lim_{k\rightarrow \infty} \lambda^{(\ell_k)}_J (\omega^{(\ell_k)})
			\;\; = \;\;
	\lim_{k\rightarrow \infty} \lambda^{(\ell_k)}_J (\omega^{(\ell_{k+1})})
			\;\; = \;\;
	\min_{\omega \in \Omega} \lambda_J(\omega).	
\end{equation}
Finally, by the continuity of $\lambda_J(\omega)$, the sequence $\{ \omega^{(\ell_k)} \}$ must converge to a 
global minimizer of $\lambda_J(\omega)$.

\textbf{(ii)} Let $\lambda_\ast := \min_{\omega \in \Omega} \lambda_J(\omega)$. We first show that the sequence
$\{ \lambda_J^{(k)}(\omega^{(k+1)}) \}$ is convergent. In this respect observe that for every pair of positive
integers $p, k$ such that $p > k$, due to monotonicity (Lemma \ref{lemma:monotonicity}) we have
\[
	\begin{split}
	\lambda_\ast \;\; & \geq \;\; \min_{\omega \in \Omega} \lambda_J^{(p)} (\omega) \;\; = \;\;  \lambda_J^{(p)}(\omega^{(p+1)}) 	\\
				\;\; & \geq \;\; \lambda_J^{(k)}(\omega^{(p+1)}) \;\; \geq \;\; \min_{\omega \in \Omega} \lambda_J^{(k)} (\omega) \;\; = \;\; \lambda_J^{(k)}(\omega^{(k+1)}).
	\end{split}
\]
Hence the sequence $\{ \lambda_J^{(k)}(\omega^{(k+1)}) \}$ is monotonically increasing bounded above by $\lambda_\ast$
proving its convergence. 

Now let $\{ \omega^{(\ell_k)} \}$ be a convergent subsequence of $\{ \omega^{(k)} \}$ as in part (i), and let us consider the 
sequence $\{ \lambda_J^{(\ell_{k+1}-1)}(\omega^{(\ell_{k+1})}) \}$, which is a
subsequence of $\{ \lambda_J^{(k)}(\omega^{(k+1)}) \}$. We complete the proof by establishing the convergence of this 
subsequence to $\lambda_\ast$.
The monotonicity and the boundedness of $\{ \lambda_J^{(k)}(\omega^{(k+1)}) \}$ from above by $\lambda_\ast$
imply
\begin{equation}\label{eq:order_seq}
	 \lambda_J^{(\ell_k)}(\omega^{(\ell_{k+1})})	\;\;	\leq	\;\;		\lambda_J^{(\ell_{k+1}-1)}(\omega^{(\ell_{k+1})})	\;\;	\leq	\;\;		\lambda_\ast.
\end{equation}
Above, as shown in part (i) (in particular see (\ref{eq:all_converge})),
$\lim_{k\rightarrow \infty} \lambda_J^{(\ell_k)}(\omega^{(\ell_{k+1})}) = \lambda_\ast$, so we must also have
 $\lim_{k\rightarrow \infty} \lambda_J^{(\ell_{k+1}-1)}(\omega^{(\ell_{k+1})}) = \lambda_\ast$ as desired.
\end{proof}

As for the maximization problem \textbf{(MX)}, it does not seem possible to conclude with such a
convergence result to a global maximizer. This is because only a lower bound (but not an upper
bound) is available in terms of the reduced eigenvalue functions for the maximum value of $\lambda_J(\omega)$
over $\omega \in \Omega$. However, the sequence $\{ \lambda_J^{(k)}(\omega^{(k+1)}) \}$ is still convergent 
as shown next.
\begin{theorem}\label{thm:global_conv_max}
The sequence $\{ \lambda_J^{(k)}(\omega^{(k+1)}) \}$ generated by Algorithm \ref{alg} and Algorithm \ref{alge}
for the maximization problem \textbf{(MX)} converges. 
\end{theorem}
\begin{proof}
Letting $\lambda^\ast := \max_{\omega \in \Omega} \lambda_J(\omega)$, the sequence by Algorithm \ref{alg} and Algorithm \ref{alge} 
satisfies
\begin{equation}\label{eq:global_conv_max}
	\begin{split}
	\lambda_J^{(k-1)} (\omega^{(k)})	\;\; & \leq \;\;	\lambda_J (\omega^{(k)})		\;\; = \;\;	\lambda_J^{(k)}(\omega^{(k)})	\\
								\;\; & \leq \;\; 	\max_{\omega \in \Omega} \: \lambda_J^{(k)} (\omega)	\;\; = \;\;	\lambda_J^{(k)}(\omega^{(k+1)})	\;\; \leq \;\; \lambda^\ast
	\end{split}
\end{equation}
where the first and the last inequality follow from the monotonicity, and the equality in the first line 
is a consequence of the interpolatory property. These inequalities lead to the conclusion that the
sequence $\{ \lambda_J^{(k)} (\omega^{(k+1)}) \}$ is monotone increasing and bounded above 
by $\lambda^\ast$, hence convergent.
\end{proof}

\noindent
We observe in practice that the sequence $\{ \lambda_J^{(k)} (\omega^{(k+1)}) \}$ converges to a $\widetilde{\lambda}$ 
such that $\lambda_J(\widetilde{\omega}) = \widetilde{\lambda}$ for some $\widetilde{\omega}$ that is a local maximizer
of $\lambda_J(\omega)$, that is not necessarily a global maximizer.

We illustrate these convergence results for the minimization as well as for the maximization of the largest eigenvalue 
$\lambda_1(\omega)$ of
\begin{equation}\label{eq:numrad}
	A(\omega)
		\;\;	=	\;\;
	\frac{Ae^{i\omega} + A^\ast e^{-i\omega}}{2}
		\;\;	=	\;\;
	{\rm cos}(\omega)
	\left(
	\frac{A + A^\ast}{2}
	\right)
		+
	{\rm sin}(\omega)
	\left(
	\frac{i A - i A^\ast}{2}
	\right)
\end{equation}
over $\omega \in [0, 2\pi]$ and for a particular matrix $A$. The maximum of the largest eigenvalue of $A(\omega)$ over $\omega \in [0, 2\pi]$ 
corresponds to the numerical radius of $A$ \cite{Horn1991, He1997}, see
Section \ref{sec:numrad} for more on the numerical radius. Here we particularly choose $A$ as the 
$400\times 400$ matrix whose real part comes from  a five-point finite difference discretization of a Poisson equation, 
and whose complex part has random entries selected independently from a normal distribution with zero mean and standard 
deviation equal to 20. An application of the basic subspace procedure (Algorithm \ref{alg})  for the maximization of $\lambda_1(\omega)$ 
starting with $\omega^{(1)} = 2$ results in convergence to a local maximizer $\widetilde{\omega} \approx 2.353$, whereas initiating
the procedure with $\omega^{(1)} = 5.5$ leads to convergence to $\omega^\ast \approx 6.145$, the unique global maximizer of 
$\lambda_1(\omega)$ over $[0, 2\pi]$. This is depicted on the top row in Figure \ref{fig:convergence} on the left and
on the right, respectively. The situation is quite different when the subspace procedure is applied to minimize $\lambda_1(\omega)$.
It converges to the global minimizer $\omega_\ast \approx 3.959$ of $\lambda_1(\omega)$ regardless whether the procedure is initiated with
$\omega^{(1)} = 2$ or $\omega^{(1)} = 5.5$ as depicted at the bottom row of Figure \ref{fig:convergence}. Notice that the subspace
procedure for the maximization problem constructs reduced eigenvalue functions that capture $\lambda_1(\omega)$ well only locally
around the maximizers. In contrast for the minimization problem the reduced eigenvalue functions capture
$\lambda_1(\omega)$ globally, but their accuracy is higher around the minimizers.

\begin{figure}[h]
		\hskip -2ex
		\begin{tabular}{ll}
		\includegraphics[width=.49\textwidth]{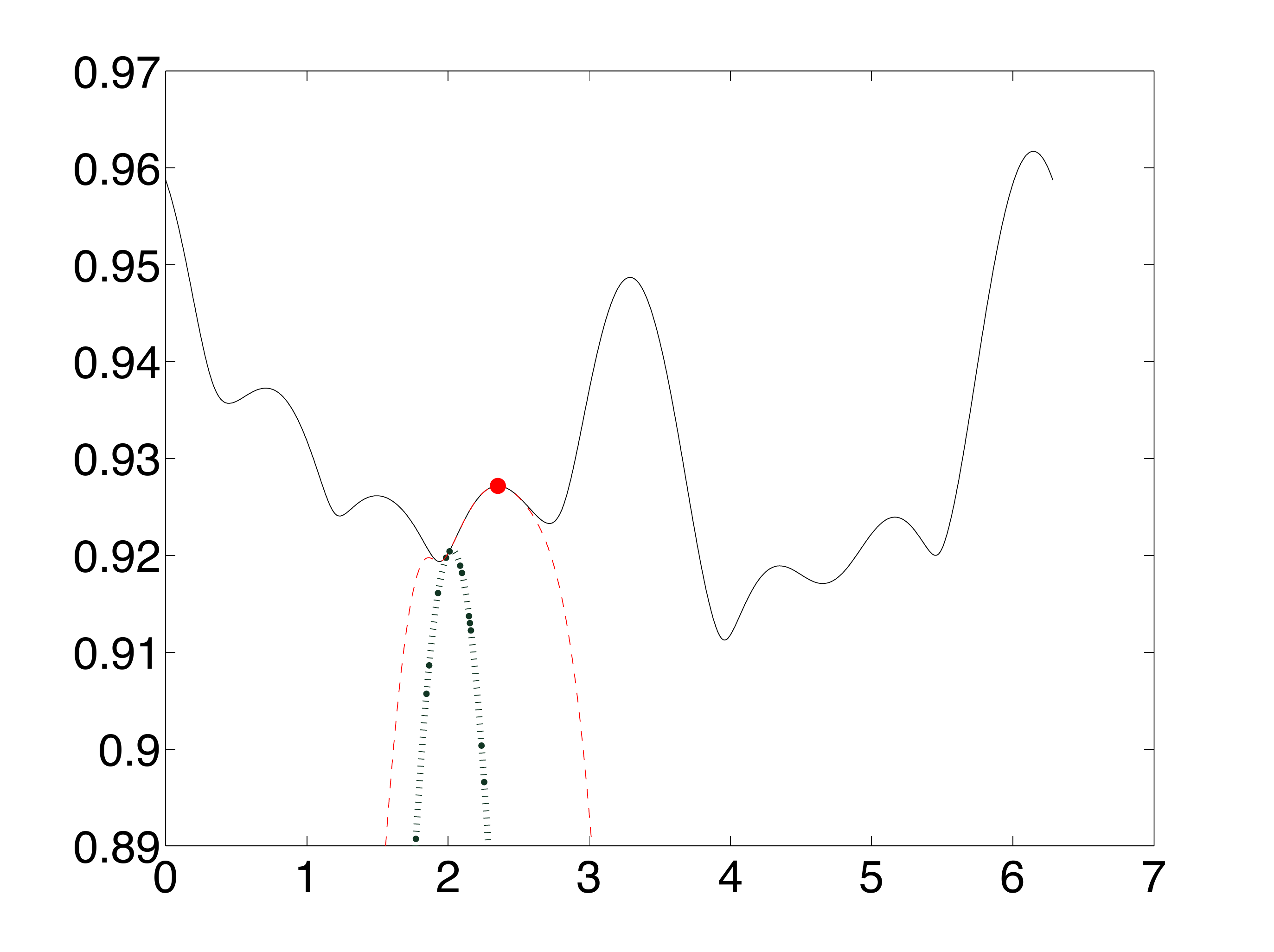} 	&
		\includegraphics[width=.49\textwidth]{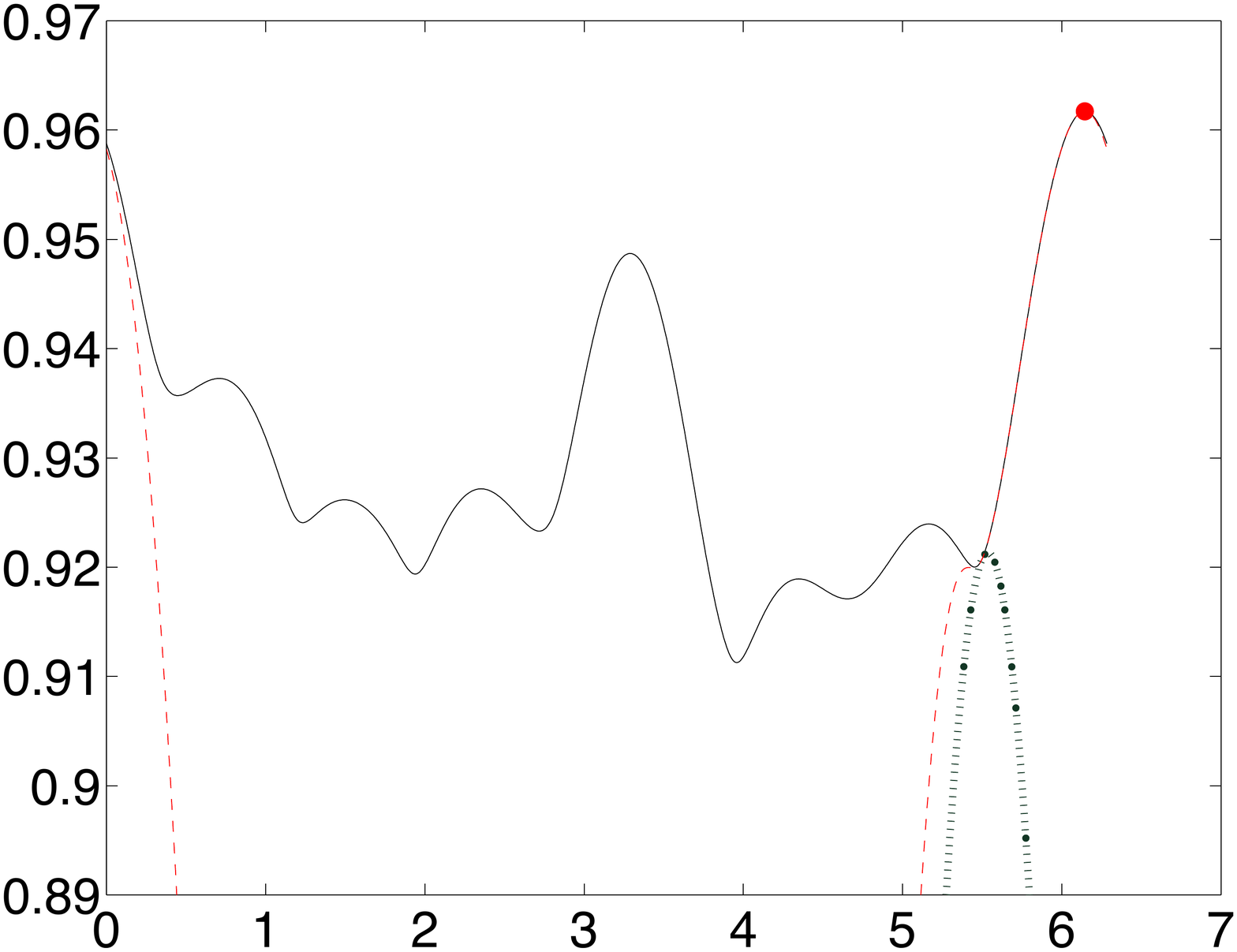} 	\\
		\includegraphics[width=.49\textwidth]{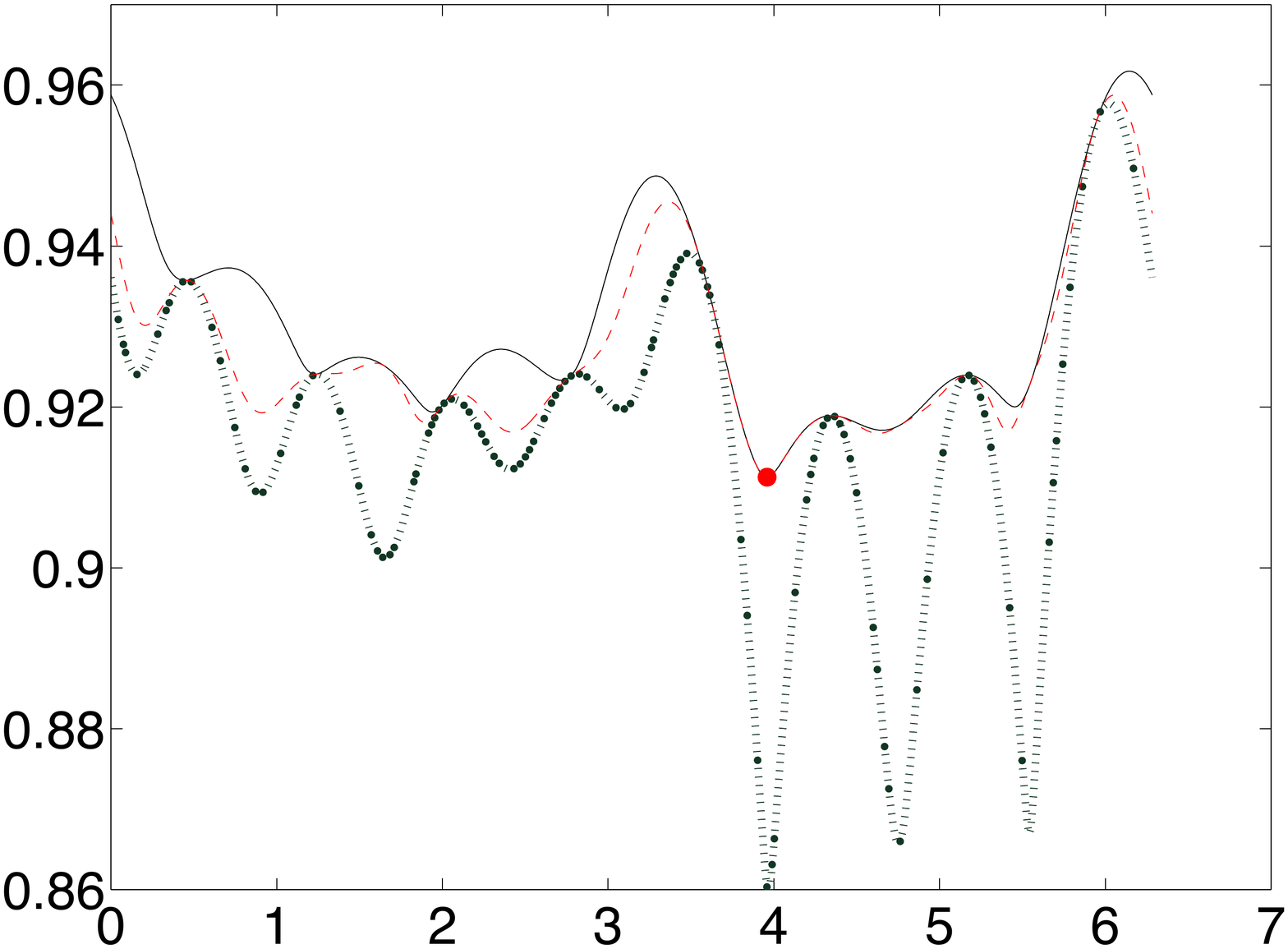} 	&
		\includegraphics[width=.49\textwidth]{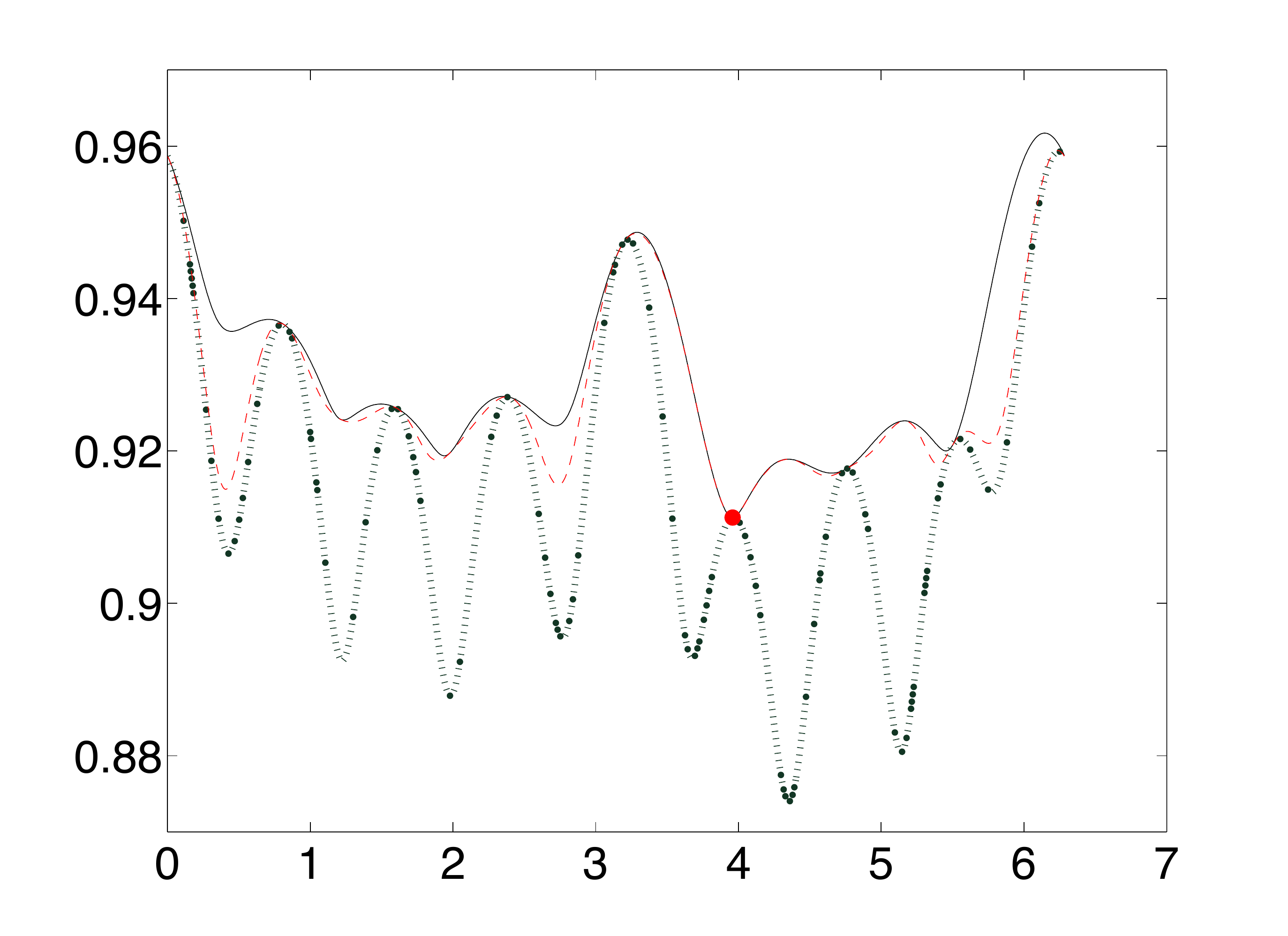}
		\end{tabular}
	  \caption{
	 Top and bottom rows illustrate Algorithm \ref{alg} to maximize and minimize, respectively, the largest eigenvalue $\lambda_1(\omega)$ (solid curves)
	 of $A(\omega)$ as in (\ref{eq:numrad}) over $\omega \in [0,2\pi]$ (the horizontal axis). For both the maximization and the minimization, Algorithm \ref{alg}
	 is initiated with two different initial points $\omega^{(1)} = 2$ on the left column and $\omega^{(1)} = 5.5$ on the right column.
	 The dotted curves on the top and at the bottom represent the reduced eigenvalue functions
	 with one and eight dimensional subspaces, respectively. The dashed curves correspond to the reduced eigenvalue
	 functions right before convergence, with five dimensional subspaces on the top and fifteen dimensional subspaces at the bottom. 
	 The dots represent the converged points, that is denoting the converged $\omega$ value with $\widehat{\omega}$ and
	 letting $\widehat{\lambda} := \lambda_1(\widehat{\omega})$ it marks $(\widehat{\omega}, \widehat{\lambda})$. 
	   	     }
	     \label{fig:convergence}
\end{figure}

\subsection{The Rate-of-Convergence}\label{sec:rate_of_convergence}

Next we are concerned with how quickly the iterates of the subspace procedures converge when they do converge
to a smooth stationary point of $\lambda_J(\omega)$. Note that by Theorem \ref{thm:global_conv}, for the minimization problem, 
if the sequence $\{ \omega^{(k)} \}$ by Algorithm \ref{alge} (or Algorithm \ref{alg}) converges to a point $\omega_\ast$ 
where $\lambda_J(\omega)$ is simple, then $\omega_\ast$ must be a smooth stationary point of $\lambda_J(\omega)$. 
The analysis below applies to Algorithm \ref{alge} (and Algorithm \ref{alg} when $d = 1$) in a unified way, both for the 
minimization problem and for the maximization problem.
\begin{theorem}[Superlinear Convergence]\label{thm:super_convergence}
Suppose that the sequence $\{ \omega^{(k)} \}$ by Algorithm \ref{alge} (or Algorithm \ref{alg} when $d = 1$)
converges to a point $\omega_\ast$ in the interior of $\Omega$ such that
\begin{center}
	\textbf{(i)} $\lambda_J(\omega_\ast)$ is simple, 
	\textbf{(ii)} $\nabla \lambda_J(\omega_\ast) = 0$, and
	\textbf{(iii)} $\nabla^2 \lambda_J(\omega_\ast)$ is invertible.
\end{center}
Then there exists a constant $\alpha > 0$ such that
\begin{equation}\label{eq:quad_conv}
\frac{ \left\| \omega^{(k+1)} - \omega_\ast \right\|_2 }
	{ \left\| \omega^{(k)} - \omega_\ast \right\|_2 \: \max \left\{ \left\| \omega^{(k)} - \omega_\ast \right\|_2,  \left\| \omega^{(k-1)} - \omega_\ast \right\|_2 \right\} }	\;\; \leq \;\;	\alpha
	\;\;\;\;  \forall k.
\end{equation}
\end{theorem}
\begin{proof}
The arguments in the first two paragraphs of the proof of Lemma \ref{eq:accuracy_sec_der} establish
the existence of a ball ${\mathcal B}(\omega_\ast,\eta)$ containing ${\mathcal B}(\omega^{(k)}, h^{(k)})$ for all $k$ large enough, say for $k \geq k'$,
such that $\lambda_J(\omega)$ as well as $\lambda_J^{(k)}(\omega)$ for $k \geq k'$ are simple for all $\omega \in {\mathcal B}(\omega_\ast,\eta)$. 
The Hessians $\nabla^2 \lambda_J(\omega)$ and $\nabla^2 \lambda_J^{(k)}(\omega)$ for $k \geq k'$
are Lipschitz continuous inside this ball. 

Additionally, following the arguments in the proof of part (ii) of Lemma \ref{eq:accuracy_sec_der},
the invertibility of $\nabla^2 \lambda_J(\omega_\ast)$ ensures that $\nabla^2 \lambda_J(\omega^{(k)})$ 
and $\nabla^2 \lambda_J^{(k)}(\omega^{(k)})$ are invertible (indeed 
$\|  [ \nabla^2 \lambda_J(\omega^{(k)}) ]^{-1} \|_2$ and $\| [ \nabla^2 \lambda_J^{(k)}(\omega^{(k)}) ]^{-1} \|_2$
are bounded from above by some constant) for all $k$ large enough, say for $k \geq k'' \geq k'$.

Now for $k \geq k''$ by the Taylor's theorem with integral remainder
\[
				0	\;\; = \;\;
	\nabla \lambda_J\left( \omega_\ast \right)
				\;\; = \;\;
	\nabla \lambda_J( \omega^{(k)} )
				\;\; + \;\; 	
	\int_0^1	\nabla^2 \lambda_J ( \omega^{(k)} + \alpha ( \omega_\ast - \omega^{(k)} ) ) \: ( \omega_{\ast} - \omega^{(k)} ) \; {\mathrm d}\alpha.
\]
Employing $\nabla \lambda_J \left( \omega^{(k)} \right) = \nabla \lambda^{(k)}_J \left( \omega^{(k)} \right)$ (Lemma \ref{thm:first_der}, part (iii))
in this equation and left-multiplying the both sides of the equation by the inverse of $\nabla^2 \lambda_J \left( \omega^{(k)} \right)$ yield
\begin{equation}\label{eq:cor_Taylor}
\begin{split}
				0	\;\; =  \;\;
	\left[ \nabla^2 \lambda_J ( \omega^{(k)} )	 \right]^{-1} \cdot \nabla \lambda^{(k)}_J ( \omega^{(k)} )
				\;\; + \;\; (\omega_\ast - \omega^{(k)}) \;\; + \;\; \hskip 21ex \\ 
	 \hskip -1ex  \left[ \nabla^2 \lambda_J ( \omega^{(k)} )	 \right]^{-1}
	 \int_0^1
		\left[
			\nabla^2 \lambda_J ( \omega^{(k)} + \alpha (\omega_\ast - \omega^{(k)}) )
						-
			\nabla^2 \lambda_J (  \omega^{(k)} )
		\right]
		( \omega_{\ast} - \omega^{(k)} ) \; {\mathrm d}\alpha.
\end{split}
\end{equation}
A second order Taylor expansion of $\nabla \lambda^{(k)}_J \left(  \omega  \right) $ about  $\omega^{(k)}$, noting $\nabla \lambda^{(k)}_J (  \omega^{(k + 1)}  ) = 0$,
implies
\[
	\left[ \nabla^2 \lambda^{(k)}_J ( \omega^{(k)} ) \right]^{-1} \nabla \lambda^{(k)} ( \omega^{(k)} )
				\;\; = \;\;
		- ( \omega^{(k + 1)} - \omega^{(k)}  ) \; + \; 
				O([h^{(k+1)}]^2).
\]
Using this equality in (\ref{eq:cor_Taylor}) leads us to
\begin{equation}\label{eq:cor_Taylor2}
	\begin{split}
	0
		\;\; = \;\; 
	 \omega_\ast - \omega^{(k+1)}  \;\; + \;\;
	 \left\{ \left[ \nabla^2 \lambda_J ( \omega^{(k)} )	\right]^{-1} 	-  
		\left[ \nabla^2 \lambda^{(k)}_J ( \omega^{(k)} )  \right]^{-1}	\right\}
		\nabla \lambda_J ( \omega^{(k)} ) \hskip 3ex \\ 
		 \;\; + \;\;   O([h^{(k+1)}]^2)   \;\; + \;\;  \hskip 35ex \\
	\hskip -1.5ex
	\left[ \nabla^2 \lambda_J ( \omega^{(k)} )	 \right]^{-1} 
	\int_0^1	
		\left[
			\nabla^2 \lambda_J ( \omega^{(k)} + \alpha (\omega_\ast - \omega^{(k)}) )
						-
			\nabla^2 \lambda_J (  \omega^{(k)} )
		\right]
	( \omega_{\ast} - \omega^{(k)} )  \: {\mathrm d}\alpha,
	\end{split}
\end{equation}
which, by taking the 2-norms and employing the triangle inequality, yields
\begin{equation}\label{eq:final_norms}
	\begin{split}
	\| \omega^{(k+1)} - \omega_\ast \|_2	 \;\;  \leq  \;\;
	\left\|  \left[ \nabla^2 \lambda_J ( \omega^{(k)} )	\right]^{-1} 	-   
		\left[ \nabla^2 \lambda^{(k)}_J ( \omega^{(k)} )  \right]^{-1} 	\right\|_2 
		\|  \nabla \lambda_J ( \omega^{(k)} )  \|_2	   \hskip 6ex \\
			+  \;\; O([h^{(k+1)}]^2)  \;\; +  \hskip 41ex \\
		 \left\| \left[ \nabla^2 \lambda_J ( \omega^{(k)} )	 \right]^{-1} \right\|_2 
	\int_0^1	
		\left\|
			\nabla^2 \lambda_J ( \omega^{(k)} + \alpha (\omega_\ast - \omega^{(k)}) )
						-
			\nabla^2 \lambda_J (  \omega^{(k)} )
		\right\|_2
	\| \omega^{(k)}  -  \omega_{\ast}  \|_2  \: {\mathrm d}\alpha.	
	\end{split}
\end{equation}

To conclude with the desired superlinear convergence result, we bound the terms on the right-hand side of the inequality in
(\ref{eq:final_norms}) from above in terms of $\| \omega^{(k+1)} - \omega_\ast \|_2$, $\| \omega^{(k)} - \omega_\ast \|_2$ and $\| \omega^{(k-1)} - \omega_\ast \|_2$.
To this end, first note that the terms in the third line of (\ref{eq:final_norms}) is $O(\| \omega^{(k)}  -  \omega_{\ast}  \|_2^2)$, 
this is because of the boundedness of  $\| [ \nabla^2 \lambda_J ( \omega^{(k)} )	 ]^{-1} \|_2$ and the Lipschitz continuity of
$\nabla^2 \lambda_J (  \omega )$ inside ${\mathcal B}(\omega_\ast,\eta)$. Secondly  
$\|  \nabla \lambda_J ( \omega^{(k)} )  \|_2 = O( \| \omega^{(k)} - \omega_\ast \| )$, as can be seen from a Taylor expansion of $\nabla \lambda_J(\omega)$ 
about $\omega^{(k)}$ and by exploiting $\nabla \lambda_J(\omega_\ast) = 0$. Finally, due to part (ii) of Lemma \ref{thm:sec_der_ext}, we have
	$
		\|    [ \nabla^2 \lambda_J ( \omega^{(k)} )	 ]^{-1} 	-   [ \nabla^2 \lambda^{(k)}_J ( \omega^{(k)} )  ]^{-1} 	\|_2
				=
			O( h^{(k)} ).
	$
Applying all these bounds to (\ref{eq:final_norms}) gives rise to
\[
	\| \omega^{(k+1)} - \omega_\ast \|_2	  \;\; \leq \;\; O(h^{(k)} \| \omega^{(k)} - \omega_\ast \|_2) + O([h^{(k+1)}]^2) + O(\| \omega^{(k)}  -  \omega_{\ast}  \|_2^2).
\]
The desired result (\ref{eq:quad_conv}) follows noting $h^{(k)} \leq 2 \max \{ \| \omega^{(k)} - \omega_\ast \|_2,  \| \omega^{(k-1)} - \omega_\ast \|_2 \}$
and $[h^{(k+1)}]^2 \leq 2 ( \| \omega^{(k)}  -  \omega_{\ast}  \|_2^2  +   \| \omega^{(k+1)}  -  \omega_{\ast}  \|_2^2 )$.
\end{proof}

\noindent
Regarding, specifically, the minimization of the largest eigenvalue when $d=1$, the order of the superlinear rate-of-convergence
for Algorithm \ref{alg} is shown to be at least $1 + \sqrt{2}$ in \cite[Theorem 1]{Kressner2017}.

It does not seem straightforward to extend the rate-of-convergence result above to Algorithm \ref{alg} when $d \geq 2$,
because relations such as the ones given by Lemma \ref{thm:sec_der_ext} between the second derivatives of 
$\lambda_J(\omega)$ and $\lambda_J^{(k)}(\omega)$ are not evident. We observe a superlinear
rate-of-convergence for Algorithm \ref{alg} in practice as well. Algorithm \ref{alge} requires a slightly
fewer iterations to reach a prescribed accuracy, but Algorithm \ref{alg} attains the prescribed accuracy
with subspaces of smaller dimension. These observations are illustrated 
in Table \ref{table:rate_conv} on the problem of minimizing the largest eigenvalue $\lambda_1(\omega)$ of
\begin{equation}\label{eq:affine}
	A(\omega)	\;\;	=	\;\;	A_0		+	\omega_1 A_1			+	\dots		+	\omega_d A_d	
\end{equation}
for given Hermitian matrices $A_0, A_1, \dots, A_d \in {\mathbb C}^{n\times n}$ for $d = 2,3$, where we 
restrict $\omega_1, \dots, \omega_d$ to the interval $[-60,60]$. Such eigenvalue optimization problems 
are convex \cite{Overton1988}, and arise from a classical structural design problem \cite{Cox1992, Lewis1996}.
Additionally, as discussed in Section \ref{sec:min_large_eig}, they are closely related to semidefinite programs.
In Table \ref{table:rate_conv} the minimal values $\lambda_1^{(k)}(\omega^{(k+1)})$ of the reduced eigenvalue functions 
$\lambda_1^{(k)}(\omega)$, as well as the error $\| \omega^{(k+1)} - \omega_\ast \|_2$ converge superlinearly with respect to $k$
both for $d = 2$ and for $d = 3$. In both cases the extended subspace procedure on the right column achieves 10 decimal digits
accuracy (in the sense that $\lambda_1(\omega_\ast) - \lambda_1^{(k)}(\omega^{(k+1)}) \leq 10^{-10}$) after 8 iterations
but with subspaces of dimension 32 and 56 for $d = 2$ and $d = 3$, respectively. On the other hand, the basic subspace procedure 
on the left column requires 11 and 15 iterations, which are also the dimensions of the subspaces constructed, 
for $d=2$ and $d = 3$, respectively,  to achieve the same accuracy. 
Observe also that the minimal values $\lambda_1^{(k)}(\omega^{(k+1)})$ seem
to converge at a rate even faster than the rate-of-decay of the errors $\| \omega^{(k+1)} - \omega_\ast \|_2$. 

\begin{center}
\begin{table}

\hskip -2ex
\begin{tabular}{cccc|cccc}
	\multicolumn{8}{c}{$d = 2$}	\\
	\hline
	\hline
	\multicolumn{4}{c}{{\sc Algorithm \ref{alg}}}			\hskip 20ex								&			\multicolumn{4}{c}{{\sc Algorithm \ref{alge}}}	\hskip 20ex \\
	$k$	&	$p$ & $\lambda_1^{(k)}(\omega^{(k+1)})$	&	$\| \omega^{(k+1)} - \omega_\ast \|_2$ 	& 	$k$	& $p$ &  $\lambda_1^{(k)}(\omega^{(k+1)})$		&	$\| \omega^{(k+1)} - \omega_\ast \|_2$	\\
	\hline
	6		&	6	&	27.8402784689			&	0.3245035160						&	3	&	12	&	7.0566934870			&	0.3112458616	\\
	7		&	7	&	27.9933586806			&	0.1134918678						&	4	&	16	&	24.9336629638			&	0.5250223403	\\
	8		&	8	&	28.4522934270			&	0.0247284535						&	5	&	20	&	27.6531948787			&	0.1830729388	\\
	9		&	9	&	28.5008552358			&	0.0007732659						&	6	&	24	&	28.4440816653			&	0.0142076774	\\
	10		&	10	&	28.5010522075			&	0.0000243843						&	7	&	28	&	28.5010327361			&	0.0001556439	\\
	11		&	11	&	28.5010523924			&	0.0000000694						&	8	&	32	&	28.5010523924			&	0.0000000344	\\
\end{tabular}

\vskip 2ex

\hskip -2ex
\begin{tabular}{cccc|cccc}
	\multicolumn{8}{c}{$d = 3$}	\\
	\hline
	\hline
	\multicolumn{4}{c}{{\sc Algorithm \ref{alg}}}		\hskip 20ex									&			\multicolumn{4}{c}{{\sc Algorithm \ref{alge}}} \hskip 20ex	\\
	$k$	&	$p$ 	&	$\lambda_1^{(k)}(\omega^{(k+1)})$	&	$\| \omega^{(k+1)} - \omega_\ast \|_2$ 	& 	$k$	&	$p$  &  $\lambda_1^{(k)}(\omega^{(k+1)})$		&	$\| \omega^{(k+1)} - \omega_\ast \|_2$	\\
	\hline
	10		&		10	&	28.1244577720			&	0.1137180915						&	3	&	21	&	22.9802867532				&	0.8996337159 \\
	11		&		11	&	28.1897386962			&	0.0482696055						&	4	&	28	&	26.1225009081				&	0.6602340340	\\
	12		&		12	&	28.2359716412			&	0.0051064698						&	5	&	35	&	27.0151485822				&	0.1737449902	\\
	13		&		13	&	28.2388852363			&	0.0005705897						&	6	&	42	&	28.0850898434				&	0.0374522435	\\
	14		&		14	&	28.2389043164			&	0.0000124871						&	7	&	49	&	28.2387293186				&	0.0002142688	\\
	15		&		15	&	28.2389043663			&	0.0000001102						&	8	&	56	&	28.2389043663				&	0.0000001112	\\
\end{tabular}


\caption{ \noindent The minimization of the largest eigenvalues of two particular matrix-valued functions of the form (\ref{eq:affine})
with two and three parameters using Algorithm \ref{alg} and Algorithm \ref{alge}. The minimal values $\lambda_1^{(k)}(\omega^{(k+1)})$ of the reduced 
eigenvalue functions, as well as the errors  $\| \omega^{(k+1)} - \omega_\ast \|$ of the minimizers are listed with respect to $k$ and 
$p := {\rm dim} \: {\mathcal S}_k$. (Top Row) $d = 2$, $A_2, A_1, A_0$ are $400\times 400$ such that $A_2$ is random symmetric, $A_1$ is random 
Hermitian, whereas $A_0$ is obtained by adding  $2A_2 - A_1$ to another random symmetric matrix; 
(Bottom Row) $d=3$, $A_3, A_2, A_1, A_0$ are $400\times 400$ such that $A_3, A_2, A_1$ are random symmetric, 
while $A_0$ is obtained from a random symmetric matrix by adding  $A_3 + 2A_2 - 3A_1$. }
\label{table:rate_conv}
\vskip -5ex
\end{table}
\end{center}

\vskip -3ex

In the one parameter case (i.e., $d=1$) Theorem \ref{thm:super_convergence} applies to Algorithm \ref{alg} as well
to establish its superlinear convergence. The convergence of Algorithm \ref{alg} on the example of 
Figure \ref{fig:convergence} (concerning the minimization or maximization of the largest eigenvalue of the matrix-valued 
function in (\ref{eq:numrad})) starting with $\omega^{(1)} = 5.5$ is depicted in Table \ref{table:rate_conv2}. For the maximization 
and minimization the iterates $\{ \omega^{(k)} \}$ converge to the global maximizer ($\approx 6.145$) and global 
minimizer ($\approx 3.959$)  of $\lambda_1(\omega)$ at a superlinear rate, which is realized earlier for the maximization 
problem. The optimal values of the reduced eigenvalue function $\lambda_1^{(k)}(\omega^{(k+1)})$ converge to the globally 
maximal value and minimal value of $\lambda_1(\omega)$ even at a faster rate. 

\begin{center}
\begin{table}

\hskip 15ex
\begin{tabular}{cccc}
\multicolumn{4}{c}{\sc Maximization}		\\
	\hline
	\hline
	$k$	&	$\lambda_1^{(k)}(\omega^{(k+1)})$	&	$| \omega^{(k+1)} - \omega_\ast |$ 	&	$\omega^{(k+1)}$	 \\
	\hline
	1	&	0.9213417656		&	0.6098999764		&	5.5354507094	\\
	2	&	0.9272678333		&	0.4809562556		&	5.6643944303	\\
	3	&	0.9452307484		&	0.2569647709		&	5.8883859150	\\
	4	&	0.9596256053		&	0.0659648186		&	6.0793858672	\\
	5	&	0.9617115893		&	0.0019013105		&	6.1434493753	\\
	6	&	0.9617265293		&	0.0000002797		&	6.1453504061	\\
\end{tabular}


\hskip 15ex
\begin{tabular}{cccc}
\multicolumn{4}{c}{\sc Minimization}		\\
	\hline
	\hline
	$k$	&	$\lambda_1^{(k)}(\omega^{(k+1)})$	&	$| \omega^{(k+1)} - \omega_\ast |$ 	&	$\omega^{(k+1)}$	 \\
	\hline
	10	&	0.8905833092		&	1.2006093051		&	5.1592214939 	\\
	11	&	0.8953986993		&	2.7480373651		&	1.2105748237	\\
	12	&	0.9039755162		&	0.2734960246		&	3.6851161642	\\
	13	&	0.9112584133		&	0.0063764350		&	3.9522357538	\\
	14	&	0.9112669429		&	0.0000006079		&	3.9586115809	\\
	15	&	0.9112669442		&	0.0000000000		&	3.9586121888	\\
\end{tabular}


\caption{ \noindent Iterates of Algorithm \ref{alg} on the example depicted on the right-hand column of 
Figure \ref{fig:convergence} involving the maximization and minimization of the largest eigenvalue of a matrix-valued 
function of the form (\ref{eq:numrad}) starting with $\omega^{(1)} = 5.5$. The optimal values $\lambda_1^{(k)}(\omega^{(k+1)})$ 
of the reduced eigenvalue functions, the errors $| \omega^{(k+1)} - \omega_\ast |$ of the optimizers, 
the optimizers $\omega^{(k+1)}$ are listed with respect to $k$.
}
\label{table:rate_conv2}
\vskip -5ex
\end{table}
\end{center}

\normalsize
\section{Variations and Extensions}\label{sec:var_ext}

\subsection{A Greedy Subspace Procedure without Past}\label{sec:subspace_nopast}
The rate-of-convergence analysis of the previous section and, in particular, the proof of Theorem \ref{thm:super_convergence},
for Algorithm \ref{alge} (for Algorithm \ref{alg} when $d=1$) makes use of the eigenvectors added into the subspace in the last iteration
(in the last two iterations) only. Hence Theorem \ref{thm:super_convergence} and its superlinear convergence assertion still hold for 
Algorithm \ref{alge} even if its line 11 is changed as
\[
	{\mathcal S}_k \; \gets \; {\rm span} \left\{  s^{(k)}_1, \dots, s^{(k)}_{J}  \right\} 
				\; \bigoplus \; \left\{ \bigoplus_{p = 1, q = p}^d {\rm span} \left\{  s^{(k)}_{1,p q}, \dots, s^{(k)}_{J, p q}  \right\} \right\},
\]
that is even if the previous subspace is completely discarded. We refer to this variant of Algorithm \ref{alge} as
\emph{Algorithm \ref{alge} without past}. Similarly for Algorithm \ref{alg} when $d = 1$ 
the superlinear convergence assertion of Theorem \ref{thm:super_convergence} still holds
if only the eigenvectors from the last two iterations are kept inside the subspace, that is
if line 7 of Algorithm \ref{alg} is replaced by
\[
	{\mathcal S}_k \; \gets \; {\rm span} \left\{  s^{(k-1)}_1, \dots, s^{(k-1)}_J  \right\} \oplus {\rm span} \left\{  s^{(k)}_1, \dots, s^{(k)}_J  \right\}.
\]
We refer to this variant of Algorithm \ref{alg} for the case $d=1$ as \emph{Algorithm \ref{alg} without past}.

The main issue with these subspace procedures without past is the convergence. In Section \ref{sec:global_convergence} the convergence of 
the sequence $\{ \lambda_J^{(k)}(\omega^{(k+1)}) \}$ is established both for the minimization problem and for the maximization problem
(by Theorem \ref{thm:global_conv} and Theorem \ref{thm:global_conv_max}, respectively).
When the eigenvectors from the past iterations are discarded, the convergence of $\{ \lambda_J^{(k)}(\omega^{(k+1)}) \}$
is not guaranteed anymore for the minimization problem, 
because the  monotonicity of $\lambda_J^{(k)}(\omega)$ with respect to $k$ is no longer true. On the other hand, all the equalities 
and inequalities in (\ref{eq:global_conv_max}) concerning the monotonicity and boundedness of $\{ \lambda_J^{(k)}(\omega^{(k+1)}) \}$
for the maximization problem can be verified to hold, so this sequence is still guaranteed to converge. 

The remarks of the previous paragraph are illustrated in Table \ref{table:nopast} 
which concerns the application of Algorithm \ref{alg} without past  to the example of Table \ref{table:rate_conv2}. 
For the maximization problem, the maximizers $\lambda_1^{(k)}(\omega^{(k+1)})$	 
of the reduced eigenvalue functions $\lambda_1^{(k)}(\omega)$ converge to the globally largest value of $\lambda_1(\omega)$ at a 
superlinear rate with respect to $k$, even though for every $k$ the subspace dimension is two. For the minimization problem,
the values in the table depict that the sequence $\{ \lambda_1^{(k)}(\omega^{(k+1)}) \}$ does not converge. 

The subspace procedures without past work effectively in practice for the maximization problem. 
On the other hand, for the minimization problem, it may fail to converge. In \cite{Kressner2017}
a convergent subspace framework making use of three dimensional subspaces is devised 
for the Crawford number computation, which involves the maximization of the smallest eigenvalue (equivalently the minimization 
of the largest eigenvalue) of a matrix-valued function depending on one variable. The three dimensional subspaces are formed of 
eigenvectors from the past iterations, but not necessarily the eigenvectors from the last three iterations. The approach is built on 
the concavity of  the smallest eigenvalue function involved and the knowledge of an interval containing the global maximizer. 
It does not seem easy to extend the ideas in \cite{Kressner2017} to the general nonconvex setting that involves the 
minimization of $\lambda_J(\omega)$.

\begin{center}
\begin{table}

\hskip 23ex
\begin{tabular}{cc||cc}
	\multicolumn{2}{c}{\sc Maximization}								&			\multicolumn{2}{c}{\sc Minimization}	 \\
	$k$	&	$\lambda_1^{(k)}(\omega^{(k+1)})$						&			$k$	&	$\lambda_1^{(k)}(\omega^{(k+1)})$	\\
	\hline
	1		&	0.9213417656									&			45				&	-0.3266668608	  \\
	2		&	0.9322945953									&			46				&	-0.5147292664	  \\
	3		&	0.9595070069									&			47				&	-0.3348522694	  \\
	4		&	0.9615567360									&			48				&	-0.4859510415  \\
	5		&	0.9617263748									&			49				&	-0.3259949226  \\
	6		&	0.9617265294									&			50				&	-0.5140279537 
\end{tabular}


\caption{ \noindent The optimal values $\lambda_1^{(k)}(\omega^{(k+1)})$
of the reduced eigenvalue functions by the variant of Algorithm \ref{alg} that keeps only the subspaces from 
the last two iterations are listed with respect to $k$ for the example of the right column of Figure \ref{fig:convergence}
(which involves the maximization and minimization of the largest eigenvalue of a matrix-valued 
function of the form (\ref{eq:numrad}) starting with $\omega^{(1)} = 5.5$).}
\label{table:nopast}
\vskip -3ex
\end{table}
\end{center}

\subsection{Subspace Procedures for Singular Value Optimization}\label{sec:singular_vals}
The ideas of the previous sections can be extended to maximize or minimize the $J$th largest
singular value $\sigma_J(\omega)$  of a compact operator 
\begin{equation}\label{eq:oper_param}
		{\mathbf B}(\omega)	:=	\sum_{\ell = 1}^\kappa		f_\ell(\omega)	{\mathbf B}_\ell
\end{equation}
over a compact subset $\Omega$ of ${\mathbf R}^d$. As before ${\mathbf B}(\omega)$ is defined over $\omega$ that belongs to an open subset 
$\overline{\Omega}$ containing the feasible region $\Omega$.
In representation (\ref{eq:oper_param}) of ${\mathbf B}(\omega)$ above $f_\ell : \overline{\Omega} \rightarrow {\mathbb R}$ is a real-analytic function and 
${\mathbf B}_\ell : \ell^2({\mathbb N}) \rightarrow \ell^2 ({\mathbb N})$ is a compact operator for $\ell = 1,\dots,\kappa$.  Once again, 
the infinite-dimensional problem is motivated by the finite dimensional case where the matrix-valued function
\begin{equation}\label{eq:mat_func2d}
		B(\omega)	:=	\sum_{\ell = 1}^\kappa		f_\ell(\omega)	B_\ell,
\end{equation}
for given $B_\ell \in {\mathbb C}^{p\times q} \; \ell = 1,\dots,\kappa$ with large dimensions,
takes the role of ${\mathbf B}(\omega)$.

Let ${\mathcal U}, {\mathcal V}$ be two subspaces of $\ell^2({\mathbb N})$ with equal dimension,
and let $U, \: V$ be orthonormal bases for these subspaces. 
The subspace procedures to optimize $\sigma_J(\omega)$ are based on the optimization of the $J$th largest 
singular value $\sigma_J^{{\mathcal U}, {\mathcal V}}(\omega)$ of an operator of the form
\begin{equation}\label{eq:oper_param_red}
	{\mathbf B}^{\mathcal U, \mathcal V}(\omega)	:=	{\mathbf U}^\ast {\mathbf B}(\omega) {\mathbf V}
\end{equation}
where ${\mathbf V}$ (${\mathbf U}$) represents the operator as in (\ref{eq:change_coor}) or (\ref{eq:change_coor2})
that maps the coordinates of a vector in ${\mathcal V}$ (${\mathcal U}$) relative to the basis $V$ ($U$) to itself. 
This operator can be written as
\[
	{\mathbf B}^{{\mathcal U}, {\mathcal V}}(\omega)	=	
					\sum_{\ell = 1}^\kappa	f_\ell(\omega)	{\mathbf B}^{{\mathcal U}, {\mathcal V}}_\ell
	\;\;\; {\rm where} \;\;\;
	{\mathbf B}^{{\mathcal U}, {\mathcal V}}_\ell := {\mathbf U}^\ast {\mathbf B}_\ell {\mathbf V},
\]
which helps to reduce computational costs.

The subspace projections and restrictions here for singular value optimization are closely related to the ones employed 
for eigenvalue optimization. Indeed $\sigma_J(\omega)$ and $\sigma_J^{{\mathcal U}, {\mathcal V}}(\omega)$ correspond to 
the $J$th largest eigenvalues of
\[
	\begin{split}
	\left[
		\begin{array}{cc}
			0					&	{\mathbf B}(\omega)	\\
			{\mathbf B}^\ast	(\omega)	&	0	\\
		\end{array}
	\right]
	\quad
	{\rm and}	\quad & \hskip 2.5ex
	\left[
		\begin{array}{cc}
			0											&	{\mathbf U}^\ast {\mathbf B}(\omega)	 {\mathbf V} \\
			{\mathbf V}^\ast {\mathbf B}^\ast(\omega) {\mathbf U}	&	0	\\
		\end{array}
	\right]	\\
		&	=  
	\left[
		\begin{array}{cc}
			{\mathbf U}^\ast	&	{\mathbf V}^\ast
		\end{array}
	\right]
	\left[
		\begin{array}{cc}
			0					&	{\mathbf B}(\omega)	\\
			{\mathbf B}^\ast	(\omega)	&	0	\\
		\end{array}
	\right]
	\left[	
		\begin{array}{c}
		{\mathbf U}	\\
		{\mathbf V}	\\
		\end{array}
	\right],
	\end{split}
\]
respectively. Hence the results in Section \ref{sec:eig_subspace}, specifically 
Lemma \ref{lemma:Lipschitz_continuity} - \ref{thm:low_rank} and Theorem \ref{thm:accuracy_rproblems},
extend to relate the singular values $\sigma_J(\omega)$ and $\sigma_J^{{\mathcal U}, {\mathcal V}}(\omega)$.
For instance, monotonicity amounts to the following: for four subspaces ${\mathcal U}_1, {\mathcal U}_2, {\mathcal V}_1, {\mathcal V}_2$
of $\ell^2({\mathbb N})$ such that 
$\; {\mathcal U}_1 \subseteq {\mathcal U}_2$ and
${\mathcal V}_1 \subseteq {\mathcal V}_2$, we have 
\[
	\sigma_J^{{\mathcal U}_1, {\mathcal V}_1}(\omega)		\quad	\leq	\quad		\sigma_J^{{\mathcal U}_2, {\mathcal V}_2}(\omega)	\quad	\leq	\quad	\sigma_J(\omega).	
\]

The extended greedy subspace procedure for singular value optimization forms ${\mathcal U}$
and ${\mathcal V}$ from the left singular vectors and right singular vectors
of ${\mathbf B}(\omega)$ at the optimizers of the reduced problems and at nearby points.
A precise description is given in Algorithm \ref{algse} below, where ${\mathcal U}_k$ and ${\mathcal V}_k$ denote
the left and the right subspace at step $k$ and $\sigma_J^{(k)}(\omega) := \sigma_J^{{\mathcal U}_k, {\mathcal V}_k}(\omega)$.
The basic greedy procedure is defined similarly by modifying lines 11 and 12 as
\begin{equation}\label{eq:algsvdb}
	{\mathcal U}_k \; \gets \; {\mathcal U}_{k-1} \oplus {\rm span} \left\{  u^{(k)}_1, \dots, u^{(k)}_{J}  \right\} 
				\;\;
				{\rm and}
				\;\;
	{\mathcal V}_k \; \gets \; {\mathcal V}_{k-1} \oplus {\rm span} \left\{  v^{(k)}_1, \dots, v^{(k)}_{J}  \right\}.
\end{equation}
In the description, by a consistent pair of left and right singular vectors $u$ and $v$ corresponding to a 
singular value $\sigma$ of ${\mathbf B}(\omega)$, we mean the vectors satisfying ${\mathbf B}(\omega) v = \sigma u$ and 
${\mathbf B}^\ast (\omega) u = \sigma v$ simultaneously.


\begin{algorithm}
 \begin{algorithmic}[1]
  \REQUIRE{ A parameter dependent compact  operator ${\mathbf B}(\omega)$ of the form (\ref{eq:mat_func2d}), 
  and a compact subset $\;\; \Omega \subset \overline{\Omega} \;\;$ of ${\mathbb R}^d$.}
\STATE $\omega^{(1)} \gets$ a random point in $\Omega$.
\STATE $u^{(1)}_1, \dots, u^{(1)}_{J}$  and $v^{(1)}_1, \dots, v^{(1)}_{J}$ $\gets $ consistent left and right singular vectors  
corresponding to $\sigma_1(\omega^{(1)}), \dots, \sigma_{J}(\omega^{(1)})$.
\STATE ${\mathcal U}_1	\;	\gets \; 
		{\rm span} \left\{  u^{(1)}_1, \dots, u^{(1)}_{J}  \right\}.
		$
		and
	${\mathcal V}_1	\;	\gets \; 
		{\rm span} \left\{  v^{(1)}_1, \dots, v^{(1)}_{J}  \right\}.
		$
	
\FOR{$k \; = \; 2, \; 3, \; \dots$}
	\STATE  
			$\omega^{(k)}$ \hskip -0.4ex 
			$\gets$ \hskip -0.4ex any $\omega_\ast \in \arg\min_{\omega \in \Omega} \sigma^{(k-1)}_J(\omega) \;\;$ for the minimization problem, or \\
			$\omega^{(k)}$ \hskip -0.4ex 
			$\gets$ \hskip -0.4ex any $\omega_\ast  \in \arg\max_{\omega \in \Omega} \sigma^{(k-1)}_J(\omega) \;\;$ for the maximization problem.
	\STATE $u^{(k)}_1, \dots, u^{(k)}_{J}$ and $v^{(1)}_1, \dots, v^{(1)}_{J}$ $  \gets $ consistent left and right
	singular vectors corresponding to $\sigma_1(\omega^{(k)}), \dots, \sigma_{J}(\omega^{(k)})$. 
	\STATE $h^{(k)} \gets \| \omega^{(k)} - \omega^{(k-1)} \|_2$	
	\FOR{$p \; = \; 1, \dots, d, \quad q \; = \; p, \dots, d$}
		\STATE $u^{(k)}_{1,p q}, \dots, u^{(k)}_{J, p q}$ and $v^{(k)}_{1,p q}, \dots, v^{(k)}_{J, p q}$  $\gets $ consistent left and right singular \\
		vectors corresponding to $\sigma_1(\omega^{(k)} + h^{(k)} e_{pq}), \dots, \sigma_J(\omega^{(k)} + h^{(k)} e_{pq})$.
	\ENDFOR
	\STATE ${\mathcal U}_k \; \gets \; {\mathcal U}_{k-1} \oplus {\rm span} \left\{  u^{(k)}_1, \dots, u^{(k)}_{J}  \right\} 
				\oplus \left\{ \bigoplus_{p = 1, q = p}^d {\rm span} \left\{  u^{(k)}_{1,p q}, \dots, u^{(k)}_{J, p q}  \right\} \right\}$.
	\STATE ${\mathcal V}_k \; \gets \; {\mathcal V}_{k-1} \oplus {\rm span} \left\{  v^{(k)}_1, \dots, v^{(k)}_{J}  \right\} 
				\oplus \left\{ \bigoplus_{p = 1, q = p}^d {\rm span} \left\{  v^{(k)}_{1,p q}, \dots, v^{(k)}_{J, p q}  \right\} \right\}$.
\ENDFOR
 \end{algorithmic}
\caption{\small The Extended Greedy Subspace Procedure for Singular Value Optimization \normalsize}
\label{algse}
\end{algorithm}

Observe that the reduced singular value function $\sigma_J^{(k)}(\omega)$
is the same as $\lambda_J^{(k)}(\omega)$ formed by Algorithm \ref{alge} when it is applied to
\begin{equation}\label{eq:sval_matfun}
	{\mathbf A}(\omega)
		:=
	\left[
		\begin{array}{cc}
			0						&	{\mathbf B}(\omega)		\\
			{\mathbf B}^\ast(\omega)		&	0
		\end{array}
	\right],
\end{equation}
so Algorithm \ref{algse} applied to ${\mathbf B}(\omega)$ and Algorithm \ref{alge} applied to ${\mathbf A}(\omega)$
lead the same sequence $\{ \omega^{(k)} \}$. Similarly, the basic greedy subspace procedure 
for singular value optimization is equivalent to Algorithm \ref{alg} operating on ${\mathbf A}(\omega)$.
All of the convergence results deduced in Section \ref{sec:convergence} for eigenvalue optimization,
in particular
\begin{enumerate}
	\item global convergence for the minimization problem (Theorem \ref{thm:global_conv}), 
	\item convergence of the sequence of maximal values of the reduced problems for the maximization problem (Theorem \ref{thm:global_conv_max}),
	\item superlinear rate-of-convergence for smooth optimizers (Theorem \ref{thm:super_convergence}), 
\end{enumerate} 
carry over to this singular value optimization setting. Some of these results in Section \ref{sec:convergence} 
are proven assuming the simplicity of $\lambda_J(\omega_\ast)$ at a particular $\omega_\ast$, which translates into
a simplicity and a positivity assumption on $\sigma_J(\omega_\ast)$ in the singular value setting.
Regarding issue 2. above, we observe in practice the convergence of $\{ \sigma_J^{(k)} (\omega^{(k+1)}) \}$
for the maximization problem to $\sigma_\ast$ such that $\sigma_J(\widetilde{\omega}) = \sigma_\ast$
for some local maximizer $\widetilde{\omega}$, that is not necessarily a global maximizer,
similar to what we observe in the eigenvalue optimization setting.

\subsection{Optimization of the $J$th Smallest Singular Value}\label{sec:small_singular_vals}
The minimum (maximum) of the $J$th smallest eigenvalue $\lambda_{-J}(\omega)$ of ${\mathbf A}(\omega)$ 
is equal to the negative of the maximum (minimum) of the $J$th largest eigenvalue of $-{\mathbf A}(\omega)$.
Hence Algorithm \ref{alg} and Algorithm \ref{alge} can be adapted to optimize $\lambda_{-J}(\omega)$.

The optimization of the $J$th smallest singular value has a different nature. In particular it cannot be converted
into an optimization problem involving the $J$th largest singular value. This is partly seen by the observation
that the $J$th smallest singular value $\sigma_{-J}(\omega)$ of an operator ${\mathbf B}(\omega)$ of the 
form (\ref{eq:oper_param}) corresponds to an eigenvalue of ${\mathbf A}(\omega)$ as in (\ref{eq:sval_matfun}),
right in the middle of its spectrum. For the restricted operator ${\mathbf B}^{{\mathcal U},{\mathcal V}}(\omega)$ defined 
as in (\ref{eq:oper_param_red}) and its $J$th smallest singular value $\sigma_{-J}^{{\mathcal U}, {\mathcal V}}(\omega)$,
monotonicity is lost as the subspaces ${\mathcal U}, {\mathcal V}$ expand. In particular
$\sigma_{-J}^{{\mathcal U},{\mathcal V}}(\omega) \geq \sigma_{-J}(\omega)$ does not necessarily hold;
the restriction of the domain of ${\mathbf B}(\omega)$ to ${\mathcal V}$ causes an increase
in the $J$th smallest singular value, whereas the projection of the range onto ${\mathcal U}$ causes
a decrease in the singular value. As a consequence of the loss of monotonicity, the interpolation 
properties do not hold anymore either.
	
A neater and theoretically more-sound approach is to employ only the restrictions from the right-hand side, that is a subspace procedure that operates on ${\mathbf B}^{\mathcal V}(\omega) := {\mathbf B}(\omega) {\mathbf V}$ and its $J$th smallest
singular value. The resulting subspace procedures are equivalent to those for eigenvalue optimization
applied to ${\mathbf B}^\ast(\omega) {\mathbf B}(\omega)$, so monotonicity and interpolation properties 
as well as all the theoretical convergence properties are regained.

\subsection{A Comparison with the Cutting-Plane Methods}\label{sec:cutting_plane}
The cutting plane method was introduced by Kelley to solve convex minimization problems \cite{Kelley1960}.
For unconstrained minimization problems, the approach under-estimates convex functions with piece-wise linear 
functions globally, solves the resulting linear program, and refines the under-estimator with the addition of the 
new linear approximation about the optimizer of the linear program \cite[Section 9.3.2]{Bonnans2006}.

Let us consider again the minimization of the largest eigenvalue $\lambda_1(\omega)$ of an affine
and Hermitian matrix-valued function
\[
	A(\omega)	\; 	=	\;	A_0 		+ 	\omega_1	A_j	+	\dots		+	\omega_d A_d
\]
for given Hermitian matrices $A_0, \dots, A_d \in {\mathbb C}^{n\times n}$. Recall that the largest eigenvalue $\lambda_1(\omega)$
is convex \cite{Overton1988}, so the cutting-plane method is applicable for its minimization as outlined in Algorithm \ref{algcut}. 
In this description $\widehat{\lambda}^{(k)}_1(\omega)$ represents the piece-wise linear under-estimator for $\lambda_1(\omega)$
and is defined by
\begin{eqnarray*}
	\widehat{\lambda}^{(k)}_1(\omega) & := &	\max_{\ell = 1,\dots,k}  \;\;
							\lambda_1\left( \widehat{\omega}^{(\ell)} \right) +  
							\nabla \lambda_1 \left( \widehat{\omega}^{(\ell)} \right)^T \left( \omega - \widehat{\omega}^{(\ell)} \right)	\\
					&	=	&
							\max_{\ell = 1,\dots,k}	\;\;	\left( \widehat{v}^{(\ell)} \right)^\ast A(\omega) \: \widehat{v}^{(\ell)},						
\end{eqnarray*}				
where $\widehat{v}^{(\ell)}$ denotes a unit eigenvector corresponding to the largest eigenvalue of $A\left( \widehat{\omega}^{(\ell)} \right)$.
Furthermore, in line 4 of the outline, $\widehat{\omega}^{(k)}$ denotes the unique minimizer of the convex function 
$\widehat{\lambda}^{(k-1)}_1(\omega)$.

\begin{algorithm}
 \begin{algorithmic}[1]
  \REQUIRE{ A matrix-valued function $A : {\mathbb R}^d \rightarrow {\mathbb C}^{n\times n}$ that is affine and Hermitian}
\STATE $\widehat{\omega}^{(1)} \gets$ a random vector in ${\mathbb R}^d$.
\STATE $\widehat{v}^{(1)} \; \gets \;$ a unit eigenvector corresponding to largest eigenvalue of $A\left( \widehat{\omega}^{(1)} \right)$.
\FOR{$k \; = \; 2, \; 3, \; \dots$}
	\STATE  $\widehat{\omega}^{(k)} \gets \arg\min_{\omega \in \Omega} \widehat{\lambda}^{(k-1)}_1(\omega) \;\;$.		
	\STATE $\widehat{v}^{(k)} \; \gets \;$ a unit eigenvector corresponding to largest eigenvalue of $A\left( \widehat{\omega}^{(k)} \right)$. 
\ENDFOR
 \end{algorithmic}
\caption{The Cutting-Plane Method to Minimize the Largest Eigenvalue}
\label{algcut}
\end{algorithm}


The basic greedy subspace procedure Algorithm \ref{alg} for the minimization of the largest eigenvalue $\lambda_1(\omega)$ 
in this finite dimensional matrix-valued setting minimizes  $\lambda^{(k)}_1(\omega)$, which, by the monotonicity property, 
under-estimates $\lambda_1(\omega)$. The sequences $\{ \omega^{(k)} \}$ and $\{ \widehat{\omega}^{(k)} \}$
generated by Algorithm \ref{alg} and Algorithm \ref{algcut} are not the same in general, but, for simplicity,
let us suppose $\omega^{(\ell)} = \widehat{\omega}^{(\ell)}$ for $\ell = 1,\dots,k$ and for some $k$. In this case,
for each $\omega \in \Omega$, we have
 \begin{eqnarray*}
 	\lambda_1(\omega) \;\;  \geq \;\;  \lambda^{(k)}_1(\omega) & \;\; = &  \max_{\alpha \in {\mathbb C}^m, \; \| \alpha \|_2 = 1} \;\; \alpha^\ast V_k^\ast A(\omega) V_k \alpha \\
											& \;\;\geq & \;\; \max_{\ell = 1,\dots,k}	\;\; \left( \widehat{v}^{(\ell)} \right)^\ast A(\omega) \: \widehat{v}^{(\ell)}
												\;\; = \;\; \widehat{\lambda}^{(k)}_1(\omega)
 \end{eqnarray*}
 where the columns of the matrix $V_k$ form an orthonormal basis for the subspace ${\rm span} \left\{ \widehat{v}^{(1)}, \dots, \widehat{v}^{(k)} \right\}$.
This illustrates that, for this special case concerning the minimization of the largest eigenvalue of an affine matrix-valued function, 
 the under-estimators used by the basic greedy subspace procedure are more accurate than those used by the cutting-plane method.

The better accuracy of the subspace procedure is apparent in Figure \ref{fig:support} for the 
matrix-valued function $A(\omega) = A_0 + \omega A_1$,
where $A_0, A_1$ are $10\times 10$ random symmetric matrices. In this figure, the function $\widehat{\lambda}_1^{(2)}(\omega)$ used by the cutting plane 
method is the maximum of two linear approximations about $\widehat{\omega}^{(1)} = 0$ and $\widehat{\omega}^{(2)} = 0.2$, 
while a two dimensional subspace is used by the subspace 
procedure with $\omega^{(1)} = 0$ and $\omega^{(2)} = 0.2$.
 
 \begin{figure}[h]
 		\hskip 4ex
		\begin{tabular}{c}
		\includegraphics[width=.85\textwidth]{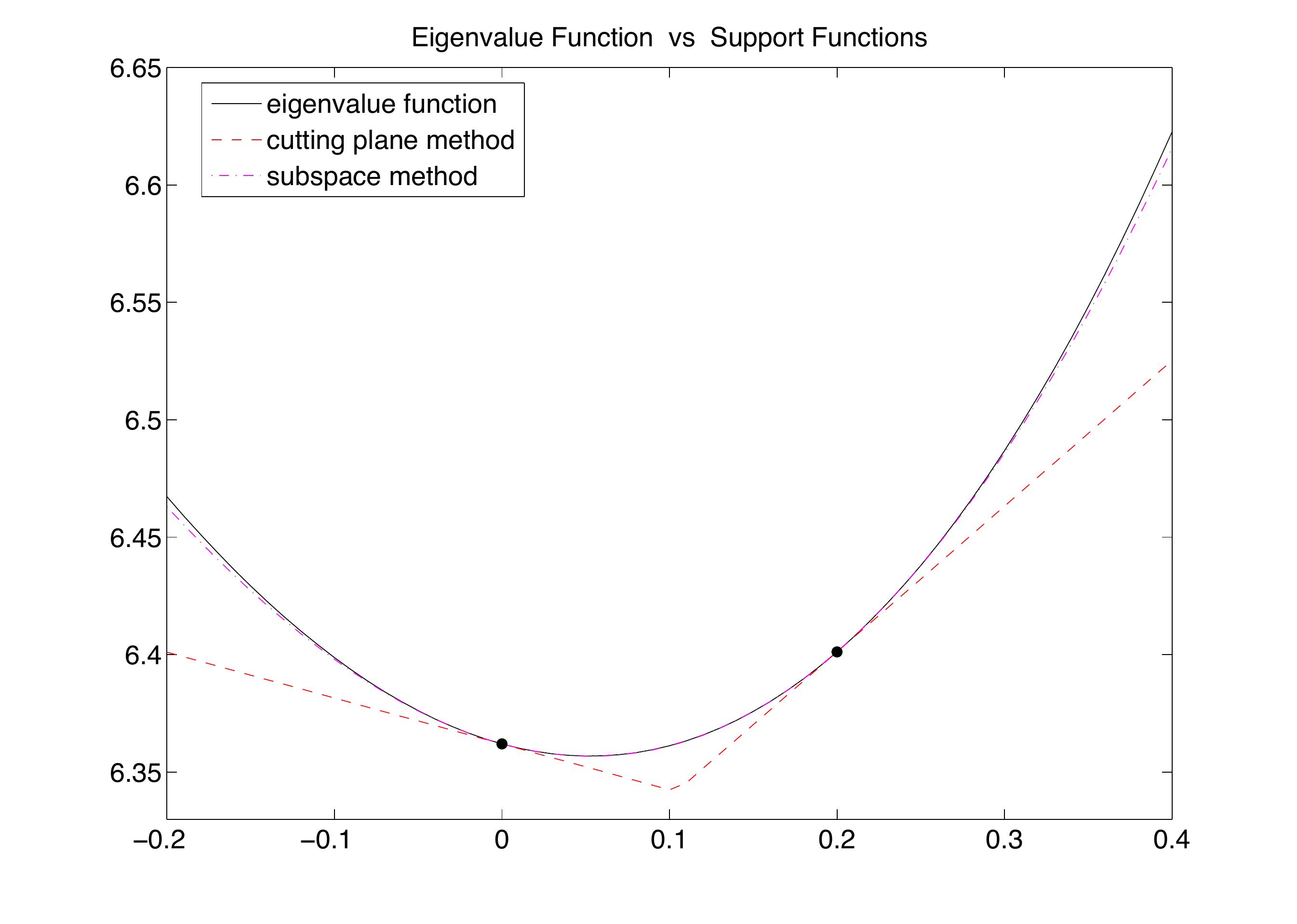} 
		\end{tabular}
	    \caption{ 
	    A comparison of the approximations by the cutting plane method and the subspace procedure:
	    (Solid Curve) The largest eigenvalue of $A(\omega) = A_0 + \omega A_1$ for randomly selected 
	    $10\times 10$ symmetric $A_0, A_1$ with respect to $\omega \in [-0.2,0.4]$;
	    (Dashed Curve) The function $\widehat{\lambda}_1^{(2)}(\omega)$ used by the cutting plane method with $\widehat{\omega}^{(1)} = 0$ and $\widehat{\omega}^{(2)} = 0.2$;
	    (Dashed-Dotted Curve) The function $\lambda_1^{(2)}(\omega)$ used by the subspace procedure with $\omega^{(1)} = 0$ and $\omega^{(2)} = 0.2$;
	    Additionally, the black dots represent $(0, \lambda_1(0))$ and $(0.2, \lambda_1(0.2))$.
	   	     }
	     \label{fig:support} 
\end{figure}


\section{Numerical Experiments}\label{sec:applications}
The next section specifies the Matlab software accompanying this work, and some important implementation
details of this software.  The rest of the section is devoted to three applications of large-scale eigenvalue and 
singular value optimization, namely, the numerical radius, the distance to instability from a matrix or a pencil, 
and the minimization of the largest eigenvalue of an affine matrix-valued function, which is closely related 
to semi-definite programming. On these examples we illustrate the power of the subspace procedures, introduced 
and analyzed in the previous sections, in practice.

\subsection{Software and Implementation Details}
MATLAB implementations accompanying this work are made publicly 
available on the internet\footnote{\texttt{http://home.ku.edu.tr/}$\sim$\texttt{emengi/software/leigopt}}.
The current version includes generic routines for the following purposes (for a prescribed
positive integer $J$):
\begin{itemize}
	\item minimization of the $J$th largest eigenvalue (based on Algorithm \ref{alg});
	\item maximization of the $J$th largest eigenvalue (based on Algorithm \ref{alg} without past if $d = 1$, or
	Algorithm \ref{alge} without past if $d > 1$);
	\item minimization of the $J$th smallest singular value (based on an adaptation of Algorithm \ref{alg} 
	without past if $d = 1$, or Algorithm \ref{alge} without past if $d > 1$, for singular value optimization). 
\end{itemize}

The Matlab routines above terminate if one of the conditions
\[
	\begin{split}
	\textbf{(1)} \; \left|	\lambda_J^{(k)}(\omega^{(k+1)})  -  \lambda_J^{(k-1)}(\omega^{(k)})	\right|	\;  \leq  	\texttt{tol}
	\;\;\: \left( {\rm or} \;\; \left|	\sigma_J^{(k)}(\omega^{(k+1)})  -  \sigma_J^{(k-1)}(\omega^{(k)}) 	\right|	\;  \leq  	\texttt{tol} \right),
				\\
	\textbf{(2)} \; \ \# \text{ subspace iterations } \; > \; \sqrt{n},	\hskip 50.8ex
	\end{split}
\] 
hold for a prescribed tolerance \texttt{tol},
where $n$ is the size of the matrices $A_\ell$ in (\ref{eq:mat_func2}) for eigenvalue optimization,
or the maximum of the dimensions $p$ and $q$ of the matrices $B_\ell$ in (\ref{eq:mat_func2d}) 
for singular value optimization. For all of the experiments in this section $\texttt{tol} = 10^{-12}$ is used 
unless otherwise specified. The second condition is never used in practice for the examples in this
section, because the other condition is fulfilled after a few subspace iterations.

The reduced eigenvalue optimization and singular value optimization problems are solved 
by means of the the MATLAB package \texttt{eigopt} \cite[Section 10]{Mengi2014}. These 
routines keep a lower bound and an upper bound for the optimal value of the reduced problem, 
and terminate when they differ by less than a prescribed tolerance. We set this tolerance
equal to $0.1 \: \texttt{tol}$. Additionally, \texttt{eigopt} requires a global lower (upper) bound $\gamma$
for the minimization (maximization) problem on the second derivatives of the eigenvalue functions 
or the singular value functions, particular choices for the three applications in this section are 
specified below. Finally, \texttt{eigopt} performs the optimization on a box, which must be supplied 
by the user. Once again, particular box choices for the applications in this section are specified below.

Large-scale eigenvalue and singular value problems are solved iteratively by means
of \texttt{eigs} and \texttt{svds} in Matlab. If the optimal values of the reduced eigenvalue or singular
value functions at two consecutive iterations are close enough (i.e., if they differ by an amount
less than $10^{-2}$), then we provide the shift $\lambda_{\rm cur} + 5 | \lambda_{\rm cur} - \lambda_{\rm pre} |$ 
($\sigma_{\rm cur} - 5 | \sigma_{\rm cur} - \sigma_{\rm pre} |$ ) - with 
$\lambda_{\rm cur}, \lambda_{\rm pre}$ ($\sigma_{\rm cur}, \sigma_{\rm pre}$)
denoting the optimal values of the current, previous reduced $J$th largest eigenvalue functions
($J$th smallest singular value functions) - to \texttt{eigs} (\texttt{svds}) 
and require it to compute the $J$ eigenvalues (singular values) closest to this shift.
Otherwise, \texttt{eigs} and \texttt{svds} are called without shifts.

\subsection{Numerical Radius}\label{sec:numrad}
The numerical radius $r(A)$ of a matrix $A \in {\mathbb C}^{n\times n}$, as also indicated at the end of Section \ref{sec:global_convergence},
is the modulus of the outermost point in the field of values \cite{Horn1991, He1997} formally defined by
\[
	r(A) 	\;\; := \;\;	\{ | z^\ast A z | \; | \; z \in {\mathbb C}^n, \; \| z \|_2 = 1 \}.
\]
This quantity is used, for instance, to analyze the convergence of the classical iterative schemes for linear systems \cite{Axelsson1994, Eiermann1993}.
Recall from Section \ref{sec:global_convergence} that it has an eigenvalue optimization characterization, 
specifically 
\[
	r(A)	\;\; = \;\; \max_{\omega \in [0, 2\pi]} \lambda_1(\omega)	\quad\quad {\rm where} \;\;		
			A(\omega)
			\;	=	\;
			\frac{Ae^{i\omega} + A^\ast e^{-i\omega}}{2}.
\]	

We apply Algorithm \ref{alg} without past as discussed in Section \ref{sec:subspace_nopast} for the computation
of the numerical radius of two family of matrices. The box and $\gamma$ supplied to \texttt{eigopt} are $[0, 2\pi]$ and
$2 \| A \|_2$, respectively. The latter is not guaranteed to be an upper bound on the second derivatives of the
eigenvalue function, but it works well in practice in our experience.

\noindent
{\bf Example 1:} The Grcar matrix is a Toeplitz matrix with 1s on the main, first, second, third superdiagonals, -1s on the first subdiagonal,
which exhibits ill-conditioned eigenvalues. It is used as a test matrix in the previous works \cite{He1997, Uhlig2009} concerning the estimation
of the numerical radius. Table \ref{table:numrad_Grcar} lists the computed values of the numerical radius by the subspace procedure
and run-times in seconds for the Grcar matrices of sizes varying between $320 - 20480$.  As the matrix sizes increase, the solution of large-scale 
eigenvalue problems take nearly all of the computation time. In contrast to this, the time spent to solve reduced eigenvalue optimization problems 
is very small and it is more or less constant as the sizes of the matrices increase. The reported values of the numerical radius for
the Grcar matrices of sizes 320 and 640 match with the results reported in \cite{Uhlig2009} up to 12 decimal digits. 

\noindent
{\bf Example 2:} This example concerns the gear matrix which is another Toeplitz matrix with 1s on the superdiagonal, subdiagonal, 1 at position $(1,n)$
and -1 at $(n,1)$. This matrix also has ill-conditioned, real eigenvalues lying in the interval $(-2,2)$. The computed values of the numerical
radius by the subspace procedure for the gear matrix of size $n$ with $n$ varying in $320 - 20480$, as well as the run-times, are given in 
Table \ref{table:numrad_Gearmat}. Once again, the computation time for increasing $n$ is dominated by large-scale eigenvalue computations.
The results reported for the gear matrices of sizes 320 and 640 match again with those reported in \cite{Uhlig2009} up to prescribed accuracy.

\begin{center}
\begin{table}
\hskip 2.5ex
\begin{tabular}{|c||ccccc|}
	\hline
	$n$		&	$\#$ iter		&	$r(A)$					&	total time		&	eigval comp			&	reduced prob	 \\
	\hline
	\hline
	320		&	11				&	3.240793870067		&	8.3 			&	0.8				&	7.4	\\
	640		&	12				&	3.241243679341		&	9.1			&	1.4				&	7.7	\\
	1280		&	13				&	3.241357030535		&	9.7			&	2.5				&	7.2	\\
	2560		&	15				&	3.241385481170		&	10.9			&	4.6				&	6.3	\\
	5120		&	16				&	3.241392607964		&	14.4			&	8.8				&	5.4	\\
	10240	&	18				&	3.241394391431		&	30.8			&	26.6				&	3.9	\\
	20480	&	19				&	3.241394837519		&	87.4			&	82.8				&	4.0	\\
	\hline
\end{tabular}


\caption{ \noindent The computed value of the numerical radius by Algorithm \ref{alg}
without past (i.e, the subspaces ${\mathcal S}_k$ are spanned by the eigenvectors at the last two iterations, hence
two dimensional), the number of subspace iterations and the run-times in seconds are listed for the Grcar matrix of size $n$ for various $n$. 
Specifically, the number of subspace iterations (2nd column) to reach the prescribed accuracy, the numerical radius (3rd column),
the total run-time (4th column), the time spent for large-scale eigenvalue computations (5th column) 
and the time spent for the solution of the 
reduced eigenvalue optimization problems (6th column) are reported w.r.t. the sizes of the matrices.
}
\label{table:numrad_Grcar}
\vskip -3ex
\end{table}
\end{center}

\begin{center}
\begin{table}
\hskip 2.5ex
\begin{tabular}{|c||ccccc|}
	\hline
	$n$		&	$\#$ iter		&	$r(A)$				&	total time				&	eigval comp		&	reduced prob	 \\
	\hline
	\hline
	320		&	5				&	1.999904217490		&	 5.2			&	0.4			&	4.8	\\
	640		&	5				&	1.999975979457		&	5.5			&	0.8			&	4.6	\\
	1280		&	6				&	1.999993985476		&	6.2			&	1.2			&	5.0	\\
	2560		&	5				&	1.999998495194		&	5.8			&	1.9			&	3.9	\\
	5120		&	5				&	1.999999623651		&	6.3			&	3.0			&	3.3	\\
	10240	&	5				&	1.999999905895		&	7.7			&	5.6			&	2.1	\\
	20480	&	5				&	1.999999976471		&	16.4			&	14.6			&	1.6	\\
	\hline
\end{tabular}


\caption{ \noindent This table lists the computed value of the numerical radius by 
Algorithm \ref{alg} without past (i.e, the subspaces ${\mathcal S}_k$ are spanned by the eigenvectors at the last two iterations, hence
two dimensional) for the gear matrix of size $n$ for various $n$, as well as the number of subspace iterations, and run-times in seconds. 
For the meaning of each column see the caption of Table \ref{table:numrad_Grcar}.
}
\label{table:numrad_Gearmat} 
\vskip -5ex
\end{table}
\end{center}

\subsection{Distance to Instability}\label{sec:distinstab}
\vskip -4ex
The distance to instability from a square matrix $A \in {\mathbb C}^{n\times n}$ with all eigenvalues
on the open left half of the complex plane, defined by
\[
	{\mathcal D}(A) \;\; := \;\;
		\min
		\{
			\| \Delta \|_2	\; | \;
			\exists \omega \in {\mathbb R},	\;\;	\det (A + \Delta - \omega {\rm i} I)	=	0			
		\},
\]
is suggested in \cite{VanLoan1985} as a measure of robust stability for the autonomous system
$x'(t) = Ax(t)$. An application of the Eckart-Young theorem \cite[Theorem 2.5.3]{Golub1996} yields the characterization
\[
	{\mathcal D}(A) \;\; = \;\;	\min_{\omega \in {\mathbb R}}	\;	\sigma_{-1} ( \omega )	
\]
where $\sigma_{-1}(\omega)$ denotes the smallest singular value of $B(\omega) = A - \omega {\rm i} I$.

Here also, we adopt Algorithm \ref{alg} without past (see Section \ref{sec:subspace_nopast}), 
noting that at step $k+1$ the two dimensional right subspace is the span of the right singular vectors 
corresponding to $\sigma_{-1}(\omega^{(k)})$ and $\sigma_{-1}(\omega^{(k-1)})$.
We illustrate numerical results on two family of sparse matrices. We set $\gamma$,
the global lower bound for the second derivatives of the singular value function for \texttt{eigopt},
equal to $-2 \| A \|_2$, which is a heuristic that works well in practice. 
The boxes supplied to \texttt{eigopt} are $[-60,60]$ and $[150, 160]$
for the first and second family, respectively. These boxes indeed contain the global minimizers.

\noindent
\textbf{Example 3: } Tolosa matrices arise from the stability analysis of an airplane. They are used as test
examples in previous works \cite{He1998, Freitag2011} concerning the computation of the distance 
to instability. These are stable matrices with all eigenvalues lying in the left half of the complex plane,
but they are nearly unstable (see Figure 4 of \cite{Freitag2011} for the spectrum of the $340\times 340$
Tolosa matrix), indeed their perturbations at a distance 0.002 have eigenvalues on the right half plane.
We run the subspace procedure on the Tolosa matrices of size 340, 1090, 2000, 4000, which are all
available through the matrix market \cite{Boisvert}. According to Table \ref{table:distinstab_Tolosa}
only four iterations suffice to reach prescribed accuracy. The time required for large-scale singular
value computations increases with respect to the size of the matrices. However the majority of the time 
is consumed for the solution of the reduced problems, but this is only because the matrices are relatively small.
The computed value of ${\mathcal D}(A)$ is the same in each case and match with the result
reported in \cite{Freitag2011} up to prescribed accuracy. \

\begin{center}
\begin{table}
\hskip 2.5ex
\begin{tabular}{|c||ccccc|}
	\hline
	$n$		&	$\#$ iter		&	${\mathcal D}(A)$		&	total time			&	singval comp			&	reduced prob	 \\
	\hline
	\hline
	340		&	4			&	0.001999796888		&	9.1 			&	1.0				&	8.0	\\
	1090		&	4			&	0.001999796888		&	9.7			&	1.4				&	8.2	\\
	2000		&	4			&	0.001999796888		&	9.9			&	1.8				&	8.1	\\
	4000		&	4			&	0.001999796888		&	10.3			&	2.7				&	7.5	\\
	\hline
\end{tabular}


\caption{ \noindent The run-times in seconds, number of iterations (2nd column) and computed values of ${\mathcal D}(A)$ (3rd column) 
are listed for Algorithm \ref{alg} without past applied to compute the distance to instability from the Tolosa matrices of size 340, 1090, 2000, 4000.
The subspaces ${\mathcal S}_k$ are two dimensional and spanned by the right singular vectors computed at the last two 
iterations. The total run-times, the time for large-scale singular value computations and the time for the reduced optimization problems
in seconds are provided in the 4th, 5th and 6th column, respectively. 
}
\label{table:distinstab_Tolosa}
\vskip -5ex
\end{table}
\end{center}

\noindent
\textbf{Example 4: } This example is taken from \cite{He1998}. A finite difference discretization of an Orr-Sommerfeld 
operator with step-size $h$ for planar Poiseuille flow leads to an $n\times n$ generalized eigenvalue problem $B_n v = \lambda L_n v$
or an $n\times n$ standard eigenvalue problem  $L_n^{-1} B_n v = \lambda v$ where  $n = 2/h - 1$,
\[
	\begin{split}
	L_n	\;	& =	\;	\frac{1}{h^2} {\rm tridiag} (1, -(2 +  h^2), 1)	\\		
	B_n	\;	& =	\;	\frac{1}{100} L_n^2	-	{\rm i} (U_n L_n + 2I),
	\end{split}
\]
and $U_n = {\rm diag} (1 - x_1^2, \dots, 1 - x_n^2)$ with $x_k = 1 + k\cdot h$ for $k = 1,\dots, n$. Matrix
$A_n = L_n^{-1} B_n$ is stable with eigenvalues on the left half plane, yet nearly unstable.

We apply the subspace procedure for the computation of the distance to instability for the Orr-Sommerfeld matrix 
$A_n$ of size $n = 400, 1000, 2000, 4000, 8000, 16000$. Note that $A_n$ is not sparse, yet 
applications of Arnoldi's method for large-scale smallest
singular value computations on $A_n - \omega {\rm i} I$ require the solutions of the linear systems of the form
$(A_n - \omega {\rm i} I) v_{k+1} = v_k$ for $v_{k+1}$ for a given $v_k$. We equivalently solve
the sparse linear system 
\begin{equation}\label{eq:OS_linsys}
	(B_n -\omega {\rm i}  L_n) v_{k+1} = L_n v_k
\end{equation}
in practice.  In Table \ref{table:distinstab_OrrSommerfeld}
we again observe that the total computation time is dominated by large-scale singular value computations as $n$ increases,
whereas the contribution of the time for the solution of the reduced problems to the total running time is very little for large $n$.
The computed values of ${\mathcal D}(A)$ in the table are listed only to six decimal digits, because for large $n$ \texttt{eigs}
could not compute singular values beyond 7-8 decimal digits in a reliable fashion. We attribute this to the fact that the norm
of $B_n$ and $L_n$ increase considerably as $n$ increases, so it is not possible to solve the linear system (\ref{eq:OS_linsys})
with high accuracy, for instance $\| B_{16000} \|_2 \approx 6.467 \cdot 10^{13}$. For smaller $n$ the computed solutions
are accurate up to 12 decimal digits, for instance the computed value of the distance to instability for $A_{400}$
by the subspace procedure is 0.001978172281 which differ from the result reported in  \cite{Freitag2011}
by an amount less than $10^{-12}$.

\begin{center}
\begin{table}
\hskip 5ex
\begin{tabular}{|c||ccccc|}
	\hline
	$n$		&	$\#$ iter		&	${\mathcal D}(A)$		&	total time			&	singval comp	&	reduced prob	 \\
	\hline
	\hline
	400		&	9			&	0.001978		&	5.0			&	0.9				&	4.1	\\
	1000		&	8			&	0.001978		&	5.3			&	1.7				&	3.4	\\
	2000		&	7			&	0.001978		&	5.9			&	3.0				&	2.8	\\
	4000		&	8			&	0.001978		&	7.9			&	5.5				&	2.3	\\
	8000		&	8			&	0.001979		&	15.2			&	11.7				&	3.4	\\
	16000	&	7			&	0.001938		&	30.8			&	27.7				&	2.9	\\
	\hline
\end{tabular}


\caption{ \noindent The number of iterations, computed values of the distance to instability and run-times are given for 
Algorithm \ref{alg} without past applied to estimate the distance to instability from the Orr-Sommerfeld matrices. 
Once again the subspaces are always two dimensional and spanned by the right singular vectors 
computed at the last two iterations. For the meaning of each column we refer to the caption of Table \ref{table:distinstab_Tolosa}.
}
\label{table:distinstab_OrrSommerfeld}
\vskip -5ex
\end{table}
\end{center}

\subsection{Minimization of the Largest Eigenvalue}\label{sec:min_large_eig} 
\vskip -4ex
A problem that drew substantial interest late 1980s and early 1990s \cite{Overton1988, Fan1995}
concerns the minimization of the largest eigenvalue of $\lambda_1(\omega)$ of
\begin{equation}\label{eq:affine_matrix}
	A(\omega)
				:=
	A_0	+ \omega_1 A_1 + \dots	+ \omega_d A_d
\end{equation}
for given symmetric matrices $A_0, \dots, A_d \in {\mathbb R}^{n\times n}$. This problem is already discussed
in Section \ref{sec:rate_of_convergence} in the more general case, when $A_0, \dots, A_d$ are complex and Hermitian.
Numerical results over there on random matrices indicate that the sequences generated by both the basic subspace 
procedure and the extended one converge at least superlinearly. An important application is in the context of semidefinite 
programming: under mild assumptions the dual of a semidefinite program can be expressed as an unconstrained 
minimization problem with the objective function $\lambda_1(\omega) + b^T \omega$ for some $b \in {\mathbb R}^d$ 
and with $\lambda_1(\omega)$ denoting the largest eigenvalue of a matrix-valued
function $A(\omega)$ of the form (\ref{eq:affine_matrix}). 

We apply the subspace procedures Algorithm \ref{alg} and Algorithm \ref{alge} for a particular notoriously difficult 
family of matrix-valued functions depending on two parameters. In this example the subspaces from the previous iterations 
are kept fully.  Note that the Matlab software is based on Algorithm \ref{alg}, additionally we apply 
Algorithm \ref{alge} for comparison purposes.  As for the parameters for \texttt{eigopt}, since the largest eigenvalue 
function is convex, $\gamma$  (the global lower bound on the second derivatives of the eigenvalue function) in 
theory can be chosen zero, instead we set $\gamma = -10^{-6}$ for numerical reliability. We have specified the box 
containing the minimizer as $[-10,10]\times [-10, 10]$.

\noindent
\textbf{Example 5:} In \cite{Overton1988}, for $A_0 \in {\mathbb R}^{n\times n}$ with its $(k,j)$ entry equal to
\[
	 \begin{cases}
                  \min \{ k , j\} & \text{if $|k - j | > 1$} \\
                  \min \{ k , j\} + 0.1 & \text{if $|k-j| = 1$}	\\
                  0	&	k = j
          \end{cases} \;\; ,
\]
the spectral radius (i.e., the absolute value of the eigenvalue furthest away from the origin) 
of $C_n(\omega) = A_0 - \sum_{j=1}^d \omega_j e_j e_j^T$ is minimized. It is observed
in that paper on the $n= 10$ case that at the optimal $\omega$, the eigenvalue with the
largest modulus has multiplicity three. This problem would correspond to a semidefinite
program relaxation of a max-cut problem, if $A_0$ had been a Laplace matrix of a graph \cite[Section 7]{Helmberg2000}.

Here we minimize the spectral radius of 
\begin{equation}\label{eq:minlargesteig}
	\widetilde{C}_n(\omega) = \frac{1}{100 n} A_0 - \omega_1 I_u - \omega_2 I_l
\end{equation}
for $n = 250, 500, 1000, 2000$, where $I_u = {\rm diag} (I_{n/2}, 0_{n/2})$ and $I_l = {\rm diag} (0_{n/2}, I_{n/2})$.
The scaling in front of $A_0$ is to make sure that the unique minimizer of the problem is inside $[-10, 10]\times [-10,10]$.
This problem can equivalently be posed as the minimization of the largest eigenvalue of
$A(\omega) := {\rm diag} (\widetilde{C}_n(\omega), -\widetilde{C}_n(\omega))$, so fits within the 
problem class described by (\ref{eq:affine_matrix}).
Table \ref{table:specrad} lists the minimal spectral radius values computed by the basic and extended
subspace procedures along with number of subspace iterations and computation time.
The total computation times are again dominated by the solutions of large eigenvalue problems.
Furthermore, even though the basic subspace procedure usually requires more iterations to reach
the prescribed accuracy, overall it solves fewer large eigenvalue problems and takes less computation time
as compared to the extended subspace procedure.

In all cases the computed minimizer $\omega_\ast$ is such that the largest eigenvalue of $A(\omega_\ast)$ 
has algebraic multiplicity three, so $\lambda_1(\omega)$ is not differentiable at $\omega_\ast$.
For instance for $n = 500$, that is when the matrix-valued function $A(\omega)$ is of size 1000,
the five largest eigenvalues of $A(\omega_\ast)$ are listed in Table \ref{table:specrad_extra}
on the left. Even the gaps between the remaining 497 positive eigenvalues are very small, as indeed among 500 
positive eigenvalues 495 of them lie in an interval of length 0.017. The fact that most of the eigenvalues belong to a small interval  
causes poor convergence properties for \texttt{eigs}. On the other hand, it appears that the nonsmoothness
does not affect the superlinear convergence of the iterates $\{ \lambda_1^{(k)}(\omega^{(k+1)}) \}$, as depicted 
in Table \ref{table:specrad_extra} on the right. This quick convergence is also apparent from the number of 
subspace iterations in Table \ref{table:specrad}. Theorem \ref{thm:super_convergence} does not apply to
this nonsmooth case. It is an open problem to come up with a formal argument explaining the quick convergence
in this nonsmooth setting.

\begin{center}
\begin{table}
\hskip 1ex
\begin{tabular}{|c||cccccc|}
	\hline
	$n$		&	$\#$ iter		&	$p$		&	$\rho_\ast$		&	total time			&	eigval comp	&	reduced prob 	 \\
	\hline
	\hline
	250		&	7			&	$\:$7$\:$	&	0.509646245274		&	3.8			&	1.5	&	2.2			\\
	500		&	7			&	7	&	1.016261471669		&	4.8 			&	2.8	&	1.6		\\
	1000		&	8			&	8	&	3.584040976076		&	13.9			&	11.8	&	1.2		 \\
	2000		&	7			&	7	&	4.055903987776		&	68.7			&	65.8	&	0.7		\\
	\hline
\end{tabular}


\hskip 1ex
\begin{tabular}{|c||cccccc|}
	\hline
	$n$		&	$\#$ iter		&	$p$	&	$\rho_\ast$		&	total time			&	eigval comp	&	reduced prob 	 \\
	\hline
	\hline
	250		&	6			&	24	&	0.509646245274		&	4.7			&	2.8		&	1.6		\\
	500		&	6			&	24	&	1.016261471669		&	7.6 			&	5.4		&	1.1		\\
	1000		&	7			&	28	&	3.584040976076		&	22.1			&	17.8		&	1.0		 \\
	2000		&	7			&	28	&	4.055903987776		&	115.9		&	106.7	&	0.8		\\
	\hline
\end{tabular}


\caption{ \noindent The table lists the number of subspace iterations, the dimension of the subspace ${\mathcal S}_k$ at termination ($p$),
computational times in seconds and the computed minimal value ($\rho_\ast$) of the spectral radius when Algorithm \ref{alg} (on the top)
and Algorithm \ref{alge} (at the bottom) are applied to minimize the largest eigenvalue of
$A(\omega) = {\rm diag} (\widetilde{C}_n(\omega), -\widetilde{C}_n(\omega))$ for the matrix-valued function
$\widetilde{C}_n(\omega)$ defined as in (\ref{eq:minlargesteig}). Note that $A(\omega)$ whose largest eigenvalue
is minimized is of size $2n$. Also note that the subspaces from all of the previous iterations are kept in these example.
}
\label{table:specrad}
\vskip -3ex
\end{table}
\end{center}

\begin{center}
\begin{table}

\begin{tabular}{cc}

\hskip 14ex

\begin{tabular}{c}

Eigenvalues of $A(\omega_\ast)$ \\
\hline
1.016261471669	\\
1.016261471669	\\
1.016261471669	\\
1.016234975093	\\
1.016234803766	\\
\end{tabular}
\hskip 4ex
&
\hskip 4ex

\begin{tabular}{c}
					\\
\begin{tabular}{|c||c|}
\hline
$k$	&	$\lambda_1^{(k)}(\omega^{(k+1)})$		\\
\hline
3	&	1.016260414001 	\\
4	&	1.016261417096	\\
5	&	1.016261471669	\\
6	&	1.016261471669	\\
\hline
\end{tabular}	
\end{tabular}	

\end{tabular}

\caption{ \noindent (Left) The five largest eigenvalues of 
$A(\omega) = {\rm diag} (\widetilde{C}_{500}(\omega), -\widetilde{C}_{500}(\omega))$ 
at $\omega_\ast$ where $\widetilde{C}_n(\omega)$ defined as in (\ref{eq:minlargesteig}) and
$\omega_\ast$ is the minimizer of the largest eigenvalue of $A(\omega)$ for $n = 500$.
(Right) The last four iterates $\{ \lambda_1^{(k)}(\omega^{(k+1)}) \}$ of  
Algorithm \ref{alg} (which keeps all of the subspaces from the previous iterations)
when the large eigenvalue computations are performed directly
by calling \texttt{eig} in MATLAB.
}
\label{table:specrad_extra}
\vskip -5ex
\end{table}
\end{center}

\section{Concluding Remarks}\label{sec:conclusion} 
\vskip -3ex
We have proposed subspace procedures to cope with large-scale eigenvalue and singular value optimization
problems. To optimize the $J$th largest eigenvalue of a Hermitian and analytic matrix-valued function $A(\omega)$ 
over $\omega$ for a prescribed integer $J$, the subspace procedures operate on a small matrix-valued function that 
acts like $A(\omega)$ in a small subspace. The subspace is expanded with the addition of the eigenvectors of 
$A(\omega)$ at the optimizer of the eigenvalue function of the small matrix-valued function and, possibly, at nearby points.
A similar strategy is adopted to optimize the $J$th largest singular value of an analytic matrix-valued function
$B(\omega)$. In that context, it is advantageous to use two different subspace restrictions
on the input and the output to $B(\omega)$ so that the resulting small matrix-valued function 
acts like the original one only in these small input and output spaces. The subspaces are
expanded with the inclusion of the left and right singular vectors of $B(\omega)$ at the
optimizers of the small problems and at nearby points. The optimization of the $J$th
smallest singular value involves some subtlety, here it seems suitable to apply restrictions
only on the input to $B(\omega)$, which is equivalent to the frameworks for eigenvalue optimization 
applied to $B(\omega)^\ast B(\omega)$.  The preferred subspace procedures for particular cases 
are summarized in the table below.

\footnotesize
\begin{center}
\begin{tabular}{|p{2.7cm}||p{9.5cm}|}
\hline
\begin{center} \textbf{Problem}	\end{center}	&	\begin{center} \textbf{Preferred Subspace Procedure} \end{center}	 \\
\hline
\hline
\begin{center}  \textbf{(1)} $\; \min_{\omega \in \Omega} \lambda_J(\omega)$ \end{center}		&		\begin{center} Algorithm \ref{alg} \end{center}	\\
\hline
\begin{center} \textbf{(2)} $\; \max_{\omega \in \Omega} \lambda_J(\omega)$ \end{center}	&		\begin{center} Algorithm \ref{alge} (Algorithm \ref{alg}) without past if $d > 1$ (if $d=1$) \end{center} \\ 
\hline
\begin{center} \textbf{(3)} $\;\: \min_{\omega \in \Omega} \sigma_J(\omega)$	\end{center}	&		\begin{center} Algorithm \ref{algse}(b)  \end{center}		\\
\hline
\begin{center} \textbf{(4)} $\;\: \max_{\omega \in \Omega} \sigma_J(\omega)$  \end{center}		&		\begin{center} Algorithm \ref{algse} (Algorithm \ref{algse}(b)) without past if $d > 1$ (if $d=1$)  \end{center} \\ 
\hline
\begin{center} \textbf{(5)} $ \min_{\omega \in \Omega} \sigma_{-J}(\omega)$ \end{center}	&		\begin{center} Algorithm \ref{alge}(s) (Algorithm \ref{alg}(s)) without past if $d > 1$ (if $d = 1$) \end{center} \\ 
\hline
\begin{center} \textbf{(6)} $ \max_{\omega \in \Omega} \sigma_{-J}(\omega)$ \end{center}	&		\begin{center} Algorithm \ref{alg}(s) \end{center}	\\
\hline
\end{tabular}
\end{center}
\normalsize

In the table, Algorithm \ref{algse}(b) refers to the basic greedy procedure for singular value optimization that uses two-sided
projections, i.e., Algorithm \ref{algse} but with lines 11 and 12 replaced by (\ref{eq:algsvdb}).
Additionally, Algorithm \ref{alg}(s) and Algorithm \ref{alge}(s) refer to the adaptations of 
Algorithm \ref{alg} and Algorithm \ref{alge} for singular value optimization, which form the subspace from the right singular 
vectors rather than the eigenvectors. Note that the minimization and maximization of a $J$th smallest eigenvalue are not 
listed in the table, since they can be posed as the maximization and minimization of a $J$th largest eigenvalue, respectively.

We have performed convergence and rate-of-convergence analyses for these subspace procedures
by extending $A(\omega)$ and $B(\omega)$ to infinite dimension, so by replacing them with compact
operators ${\mathbf A}(\omega) : \ell^2({\mathbb N}) \rightarrow \ell^2({\mathbb N})$ and 
${\mathbf B}(\omega) : \ell^2({\mathbb N}) \rightarrow \ell^2({\mathbb N})$, former of which is also
self-adjoint. Most remarkably, Theorem \ref{thm:global_conv} establishes global convergence for
Algorithms \ref{alg}-\ref{alge} when the $J$th largest eigenvalue is minimized, and
Theorem \ref{thm:super_convergence} establishes a superlinear rate-of-convergence 
for Algorithm \ref{alg} when $d = 1$ and Algorithm \ref{alge} for minimizing and
maximizing the $J$th largest eigenvalue. The superlinear convergence result is established 
under the simplicity assumption on the $J$th largest eigenvalue at the optimizer, even though we observe 
superlinear convergence in numerical experiments (e.g., see Example 5 in Section \ref{sec:min_large_eig}) 
where this simplicity assumption is violated. The convergence results do also extend
to Algorithm \ref{algse} for the optimization
of the $J$th largest singular value. The convergence properties of the proposed subspace frameworks 
for the six problems in the table above are as follows: 
 \textbf{(1)$\&$(3)$\&$(6)} Proven global convergence at a proven (observed) superlinear rate if $d = 1$ ($d > 1$);
 \textbf{(2)$\&$(4)$\&$(5)} Proven convergence but necessarily to a global solution, observed local convergence 
 at a proven superlinear rate. 

Two additional problems where the subspace procedures and their convergence analyses developed here may be applicable are
large-scale sparse estimation problems and semidefinite programs. For instance, sparse estimation problems that can be 
cast as nuclear norm minimization problems \cite{Recht2010} seem worth exploring because of their connection with
singular value optimization. As noted in the text, the dual of a standard semidefinite program can often be cast as
an eigenvalue optimization problem involving the minimization of the largest eigenvalue. 
A systematic integration of the subspace frameworks proposed here for eigenvalue optimization into large-scale
semidefinite programs, motivated by their theoretical convergence properties, is a direction that is worth investigating.

\noindent
\textbf{Acknowledgement.} We are grateful to the two anonymous referees, Daniel Szyld and Daniel Kressner 
for valuable comments on this manuscript.

\bibliography{largescale_eigopt}

\end{document}